\documentclass[a4paper]{article}
\usepackage{amsmath,amsthm,amsfonts,thmtools}
\usepackage{amssymb}
\usepackage[cmintegrals]{newtxmath}
\usepackage{graphicx}
\usepackage{listings}
\usepackage[font=footnotesize]{caption}
\usepackage[font=scriptsize]{subcaption}
\usepackage{euscript}
\usepackage{units}
\usepackage{enumerate}
\usepackage{enumitem}
\usepackage{xcolor}
\usepackage{todonotes}
\usepackage{booktabs}
\usepackage{tikz}
\usepackage{comment}
 
\DeclareMathAlphabet{\pazocal}{OMS}{zplm}{m}{n}

\usepackage[bookmarks=true,hidelinks]{hyperref}
\usepackage{bbm}
\usepackage{multirow}

\usepackage{mathtools}
\usepackage{subcaption} 
\usepackage{graphicx}
\usepackage{pgfplots}
\usepackage{float}
\usepackage{algpseudocode}
\usepackage{bbold}
\usepackage{cite}
\usepackage{microtype}
\usepackage{rotating}
\usepackage{enumitem}
\usepackage{bm}
\usepackage{hyperref}

\numberwithin{equation}{section}
\newtheorem{theorem}{Theorem}[section]

\newtheorem{algorithm}[theorem]{Algorithm}

\numberwithin{equation}{section}

\theoremstyle{definition}

\newtheoremstyle{myremarkstyle}{}{}{}{}{\bfseries}{.}{ }{}
\theoremstyle{myremarkstyle}
\declaretheorem[name=Remark,qed=$\blacksquare$,numberlike=theorem]{remark}

\makeatletter
\newcommand*{\intavg}{%
  \mint@l{-}{}%
}
\newcommand*{\mint@l}[2]{%
  \@ifnextchar\limits{%
    \mint@l{#1}%
  }{%
    \@ifnextchar\nolimits{%
      \mint@l{#1}%
    }{%
      \@ifnextchar\displaylimits{%
        \mint@l{#1}%
      }{%
        \mint@s{#2}{#1}%
      }%
    }%
  }%
}
\newcommand*{\mint@s}[2]{%
  \@ifnextchar_{%
    \mint@sub{#1}{#2}%
  }{%
    \@ifnextchar^{%
      \mint@sup{#1}{#2}%
    }{%
      \mint@{#1}{#2}{}{}%
    }%
  }%
}
\def\mint@sub#1#2_#3{%
  \@ifnextchar^{%
    \mint@sub@sup{#1}{#2}{#3}%
  }{%
    \mint@{#1}{#2}{#3}{}%
  }%
}
\def\mint@sup#1#2^#3{%
  \@ifnextchar_{%
    \mint@sub@sup{#1}{#2}{#3}%
  }{%
    \mint@{#1}{#2}{}{#3}%
  }%
}
\def\mint@sub@sup#1#2#3^#4{%
  \mint@{#1}{#2}{#3}{#4}%
}
\def\mint@sup@sub#1#2#3_#4{%
  \mint@{#1}{#2}{#4}{#3}%
}
\newcommand*{\mint@}[4]{%
  \mathop{}%
  \mkern-\thinmuskip
  \mathchoice{%
    \mint@@{#1}{#2}{#3}{#4}%
        \displaystyle\textstyle\scriptstyle
  }{%
    \mint@@{#1}{#2}{#3}{#4}%
        \textstyle\scriptstyle\scriptstyle
  }{%
    \mint@@{#1}{#2}{#3}{#4}%
        \scriptstyle\scriptscriptstyle\scriptscriptstyle
  }{%
    \mint@@{#1}{#2}{#3}{#4}%
        \scriptscriptstyle\scriptscriptstyle\scriptscriptstyle
  }%
  \mkern-\thinmuskip
  \int#1%
  \ifx\\#3\\\else_{#3}\fi
  \ifx\\#4\\\else^{#4}\fi  
}
\newcommand*{\mint@@}[7]{%
  \begingroup
    \sbox0{$#5\int\m@th$}%
    \sbox2{$#5\int_{}\m@th$}%
    \dimen2=\wd0 %
    \let\mint@limits=#1\relax
    \ifx\mint@limits\relax
      \sbox4{$#5\int_{\kern1sp}^{\kern1sp}\m@th$}%
      \ifdim\wd4>\wd2 %
        \let\mint@limits=\nolimits
      \else
        \let\mint@limits=\limits
      \fi
    \fi
    \ifx\mint@limits\displaylimits
      \ifx#5\displaystyle
        \let\mint@limits=\limits
      \fi
    \fi
    \ifx\mint@limits\limits
      \sbox0{$#7#3\m@th$}%
      \sbox2{$#7#4\m@th$}%
      \ifdim\wd0>\dimen2 %
        \dimen2=\wd0 %
      \fi
      \ifdim\wd2>\dimen2 %
        \dimen2=\wd2 %
      \fi
    \fi
    \rlap{%
      $#5%
        \vcenter{%
          \hbox to\dimen2{%
            \hss
            $#6{#2}\m@th$%
            \hss
          }%
        }%
      $%
    }%
  \endgroup
}

\def\XXint#1#2#3{{\setbox0=\hbox{$#1{#2#3}{\int}$ }
		\vcenter{\hbox{$#2#3$ }}\kern-.6\wd0}}

\renewcommand{\geq}{\geqslant}
\renewcommand{\leq}{\leqslant}

\newcommand{\eps} {\varepsilon}

\renewcommand{\epsilon}{\varepsilon}
\renewcommand{\phi}{\varphi}

\newcommand{\R}{\mathbb{R}}
\newcommand{\N}{\mathbb{N}}

\newcommand{\bu}{{\bf u}}		

\newcommand{\f}{\mathbf{f}}

\newcommand{\train}{\mathcal{S}}

\newcommand{\er}{\EuScript{E}}

\newcommand{\df}{\EuScript{D}}
\newcommand{\dom}{\mathbb{D}}

\newcommand{\res}{\EuScript{R}}

\begin{document}


\title{Physics Informed Neural Networks (PINNs) \\ for approximating nonlinear dispersive PDEs.}

\author{Genming Bai \footnotemark[1] 
	\and Ujjwal Koley \footnotemark[2]
	\and Siddhartha Mishra \footnotemark[1]
	\and Roberto Molinaro \footnotemark[1]
}

\date{\today}

\maketitle

\medskip
\centerline{$^*$ Seminar for Applied Mathematics (SAM), D-Math}
\centerline{ETH Z\"urich, R\"amistrasse 101.}
\centerline{gbai@student.ethz.ch, siddhartha.mishra@sam.math.ethz.ch,} \centerline{roberto.molinaro@sam.math.ethz.ch}

\medskip
\centerline{$^\dagger$ Centre for Applicable Mathematics, Tata Institute of Fundamental Research}
\centerline{P.O. Box 6503, GKVK Post Office, Bangalore 560065, India}\centerline{ujjwal@math.tifrbng.res.in}

\begin{abstract}
We propose a novel algorithm, based on physics-informed neural networks (PINNs) to efficiently approximate solutions of nonlinear dispersive PDEs such as the KdV-Kawahara, Camassa-Holm and Benjamin-Ono equations. The stability of solutions of these dispersive PDEs is leveraged to prove rigorous bounds on the resulting error. We present several numerical experiments to demonstrate that PINNs can approximate solutions of these dispersive PDEs very accurately. 
\end{abstract}

\section{Introduction}
Deep learning i.e., the use of deep neural networks for regression and classification, has been very successful in many different contexts in science and engineering \cite{DLnat}. These include image analysis, natural language understanding, game intelligence and protein folding. As deep neural networks are universal function approximators, it is natural to employ them as ansatz spaces for solutions of ordinary and partial differential equations, paving the way for their successful use in scientific computing. A very incomplete list of examples where deep learning is used for the numerical solutions of differential equations includes the solution of high-dimensional linear and semi-linear parabolic partial differential equations \cite{HEJ1,E1} and references therein, and for many-query problems such as those arising in uncertainty quantification (UQ), PDE constrained optimization and (Bayesian) inverse problems. Such problems can be recast as parametric partial differential equations and the use of deep neural networks in their solution is explored for elliptic and parabolic PDEs in \cite{OSZ2019,Kuty}, for transport PDEs \cite{PP1} and for hyperbolic and related PDEs \cite{DRM1,LMR1,LMRP1, LMM}, and as operator learning frameworks in \cite{DeepOnets,LMK1,Stu1,Stu2} and references therein. All the afore-mentioned methods are of the \emph{supervised learning} type \cite{DLbook} i.e., the underlying deep neural networks have to be trained on \emph{data}, either available from measurements or generated by numerical simulations. 

However, there are several interesting problems for PDEs where generating training data might be very expensive. A different strategy might be relevant for such problems, namely the so-called \emph{Physics informed neural networks} (PINNs) which collocate the PDE residual on \emph{training points} of the approximating deep neural network, thus obviating the need for generating training data. Proposed originally in \cite{DPT,Lag2,Lag1}, PINNs have been revived and developed in significantly greater detail recently in the pioneering contributions of Karniadakis and collaborators. PINNs have been successfully applied to simulate a variety of forward and inverse problems for PDEs, see \cite{KAR8, jag1, jag2, KAR9,KAR5,KAR6,KAR7,KAR1,KAR2,KAR4,shukla,MM3} and references therein. 

In a recent paper \cite{MM1}, the authors obtain rigorous estimates on the error due to PINNs for the forward problem for a variety of linear and non-linear PDEs, see \cite{MM2} for similar results on inverse problems and \cite{DAR1} for a different perspective on error estimates for PINNs. Following \cite{MM1}, one can expect that PINNs could be efficient at approximating solutions of nonlinear PDEs as long as classical solutions to such PDEs exist and are \emph{stable} in a suitable sense. So far, PINNs have only been proposed and tested for a very small fraction of PDEs. It is quite natural to examine whether they can be efficient at approximating other types of PDEs and in particular, if the considerations of \cite{MM1} apply to these PDEs, then can one derive rigorous error estimates for PINNs ?

In this paper, we investigate the utility of PINNs for approximately a large class of PDEs which arises in physics i.e., non-linear dispersive equations that model different aspects of shallow water waves \cite{Lannes}. These include the famous Korteweg-De Vries (KdV) equation and its high-order extension, the so-called Kawahara equation, the well-known Camassa-Holm type equations and the Benjamin-Ono equations. All these PDEs have several common features, namely
\begin{itemize}
    \item They model dispersive effects in shallow-water waves. 
    \item The interesting dynamics of these equations results from a balance between non-linearity and dispersion. 
    \item They are completely integrable and contain interesting structures such as interacting solitons in their solutions. 
    \item Classical solutions and their stability have been extensively investigated for these equations. 
    \item Standard numerical methods, such as finite-difference \cite{Hol1, Car, Nav, Hol2, Hol3, koley1} and finite-element \cite{koley2, koley4} for approximating these equations can be very expensive computationally. In particular, it can be very costly to obtain low errors due to the high-order (or non-local) derivatives in these equations leading to either very small time-steps for explicit methods or expensive non-linear (or linear) solvers for implicit methods.  
\end{itemize}

Given these considerations, it is very appealing to investigate if PINNs can be successfully applied for efficiently approximating these nonlinear dispersive PDEs. To this end, we adapt the PINNs algorithm to this context in this paper and prove error estimates for PINNs, leveraging the stability of underlying classical solutions into error bounds. Moreover, we perform several numerical experiments for the KdV, Kawahara, generalized Camassa-Holm and Benjamin-Ono equations to ascertain that PINNs can indeed approximate dispersive equations to high-accuracy, at low computational cost. 

The rest of the paper is organized as follows; in section \ref{sec:2}, we briefly recall the PINNs algorithm for PDEs and apply to the KdV-Kawahara PDE in section \ref{sec:3}, generalized Camassa-Holm equations in section \ref{sec:4} and the Benjamin-Ono equations in section \ref{sec:5}.

\section{Physics Informed Neural Networks}
\label{sec:2}
In this section, we follow the recent paper \cite{MM1} and briefly recapitulate the essentials of PINNs for the following abstract PDE,
\subsection{The underlying abstract PDE}
\label{sec:21}
Let $X,Y$ be separable Banach spaces with norms $\| \cdot \|_{X}$ and $\|\cdot\|_{Y}$, respectively. For definiteness, we set $Y = L^p(\dom;\R^m)$ and $X= W^{s,q}(\dom;\R^m)$, for $m \geq 1$, $1 \leq p,q < \infty$ and $s \geq 0$, with $\dom \subset \R^{\bar{d}}$, for some $\bar{d} \geq 1$. In the following, we only consider space-time domains with $\dom = (0,T) \times D \subset \R$, resulting in $\bar{d} = 2$. Let $X^{\ast} \subset X$ and $Y^{\ast} \subset Y$ be closed subspaces with norms $\|\cdot \|_{X^{\ast}}$ and $\|\cdot\|_{Y^{\ast}}$, respectively.

We start by considering the following abstract formulation of our underlying PDE:
\begin{equation}
\label{eq:pde}
\df(\bu) = \f.
\end{equation}
Here, the \emph{differential operator} is a mapping, $\df: X^{\ast} \mapsto Y^{\ast}$ and the \emph{input} $\f \in Y^{\ast}$, such that 
\begin{equation}
\label{eq:assm1}
\begin{aligned}
&(H1): \quad \|\df(\bu)\|_{Y^{\ast}} < +\infty, \quad \forall~ \bu \in X^{\ast}, ~{\rm with}~\|\bu\|_{X^{\ast}} < +\infty. \\
&(H2):\quad \|\f\|_{Y^{\ast}} < +\infty. 
\end{aligned}
\end{equation}
Moreover, we assume that for all $\f \in Y^{\ast}$, there exists a unique $\bu \in X^{\ast}$ such that \eqref{eq:pde} holds. 
\subsection{Quadrature rules}
\label{sec:22}
In the following section, we need to consider approximating integrals of functions. Hence, we need an abstract formulation for quadrature. To this end, we consider a mapping $g: \dom \mapsto \R^m$, such that $g \in Z^{\ast} \subset Y^{\ast}$. We are interested in approximating the integral,
$$
\overline{g}:= \int\limits_{\dom} g(y) dy,
$$
with $dy$ denoting the $\bar{d}$-dimensional Lebesgue measure. In order to approximate the above integral by a quadrature rule, we need the quadrature points $y_{i} \in \dom$ for $1 \leq i \leq N$, for some $N \in \N$ as well as weights $w_i$, with $w_i \in \R_+$. Then a quadrature is defined by,
\begin{equation}
\label{eq:quad}
\overline{g}_N := \sum\limits_{i=1}^N w_i g(y_i),
\end{equation}
for weights $w_i$ and quadrature points $y_i$. We further assume that the quadrature error is bounded as,
\begin{equation}
\label{eq:assm3}
\left|\overline{g} - \overline{g}_N\right| \leq C_{quad}
\left(\|g\|_{Z^{\ast}},\bar{d} \right) N^{-\alpha},
\end{equation}
for some $\alpha > 0$. 
\subsection{PINNs} 
\label{sec:23}
\subsubsection{Neural Networks.}
As PINNs are neural networks, we start a very concise description of them. Given an input $y \in \dom$, a feedforward neural network (also termed as a multi-layer perceptron), shown in figure \ref{fig:1}, transforms it to an output, through multiple layers of units (neurons) which compose of either affine-linear maps between units (in successive layers) or scalar non-linear activation functions within units, resulting in the representation,
\begin{equation}
\label{eq:ann1}
\bu_{\theta}(y) = C_K \circ\sigma \circ C_{K-1}\ldots \ldots \ldots \circ\sigma \circ C_2 \circ \sigma \circ C_1(y).
\end{equation} 
Here, $\circ$ refers to the composition of functions and $\sigma$ is a scalar (non-linear) activation function. Examples for the activation function $\sigma$ in \eqref{eq:ann1} include the sigmoid function, the hyperbolic tangent function and the \emph{ReLU} function.

For any $1 \leq k \leq K$, we define
\begin{equation}
\label{eq:C}
C_k z_k = W_k z_k + b_k, \quad {\rm for} ~ W_k \in \R^{d_{k+1} \times d_k}, z_k \in \R^{d_k}, b_k \in \R^{d_{k+1}}.
\end{equation}
For consistency of notation, we set $d_1 = \bar{d}$ and $d_K = m$. 

Our neural network \eqref{eq:ann1} consists of an input layer, an output layer and $(K-1)$ hidden layers for some $1 < K \in \N$. The $k$-th hidden layer (with $d_k$ neurons) is given an input vector $z_k \in \R^{d_k}$ and transforms it first by an affine linear map $C_k$ \eqref{eq:C} and then by a nonlinear (component wise) activation $\sigma$. A straightforward addition shows that our network contains $\left(\bar{d} + m + \sum\limits_{k=2}^{K-1} d_k\right)$ neurons. 
We also denote, 
\begin{equation}
\label{eq:theta}
\theta = \{W_k, b_k\},~ \theta_W = \{ W_k \}\quad \forall~ 1 \leq k \leq K,
\end{equation} 
to be the concatenated set of (tunable) weights for our network. It is straightforward to check that $\theta \in \Theta \subset \R^M$ with
\begin{equation}
\label{eq:ns}
M = \sum\limits_{k=1}^{K-1} (d_k +1) d_{k+1}.
\end{equation}

\begin{figure}[htbp]
\centering
\includegraphics[width=8cm]{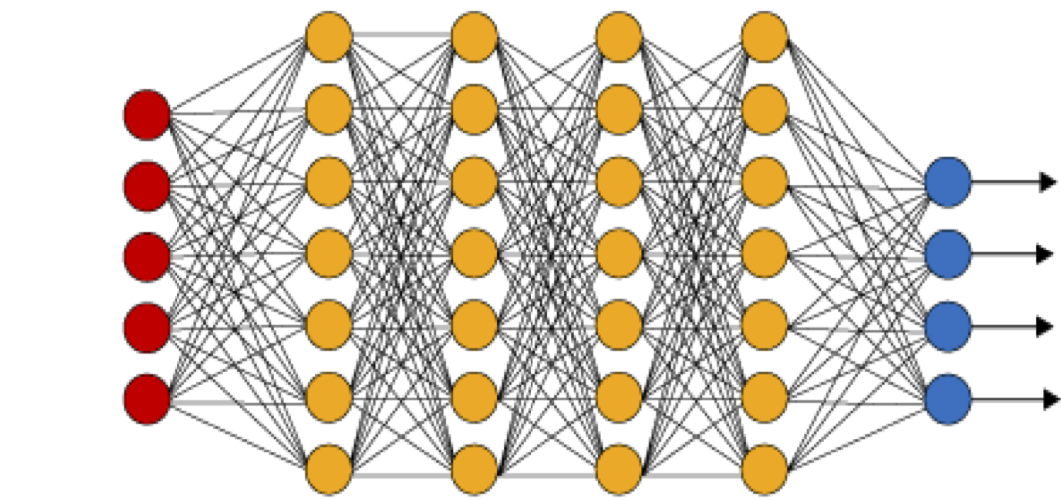}
\caption{An illustration of a (fully connected) deep neural network. The red neurons represent the inputs to the network and the blue neurons denote the output layer. They are
connected by hidden layers with yellow neurons. Each hidden unit (neuron) is connected by affine linear maps between units in different layers and then with nonlinear (scalar) activation functions within units.}
\label{fig:1}
\end{figure}

\subsubsection{Training PINNs: Loss functions and optimization}
The neural network $\bu_{\theta}$ \eqref{eq:ann1} depends on the tuning parameter $\theta \in \Theta$ of weights and biases. Within the standard paradigm of deep learning \cite{DLbook}, one \emph{trains} the network by finding tuning parameters $\theta$ such that the loss (error, mismatch, regret) between the neural network and the underlying target is minimized. Here, our target is the solution $\bu \in X^{\ast}$ of the abstract PDE \eqref{eq:pde} and we wish to find the tuning parameters $\theta$ such that the resulting neural network $\bu_{\theta}$ approximates $\bu$. 

Following standard practice of machine learning, one obtains training data $\bu(y)$, for all $y \in \train$, with training set $\train \subset \dom$ and then minimizes a loss function of the form $\sum\limits_{\train} \|\bu(y) - \bu_{\theta}(y)\|_{X}$ to find the neural network approximation for $\bu$. However, obtaining this training data requires possibly expensive numerical simulations of the underlying PDE \eqref{eq:pde}. In order to circumvent this issue, the authors of \cite{Lag1} suggest a different strategy. An abstract paraphrasing of this strategy runs as follows: we assume that for every $\theta \in \Theta$, the neural network $\bu_{\theta} \in X^{\ast}$ and $\|\bu_{\theta} \|_{X^{\ast}} < +\infty$. We define the following \emph{residual}:
\begin{equation}
    \label{eq:res1}
    \res_{\theta} = \res(\bu_\theta):= \df\left(\bu_{\theta}\right) - \f. 
\end{equation}
By assumptions (H1),(H2) (cf. \eqref{eq:assm1}), we see that $\res_{\theta} \in Y^{\ast}$ and $\|\res_{\theta}\|_{Y^{\ast}} < +\infty$ for all $\theta \in \Theta$. Note that $\res(\bu) = \df(\bu) - \f \equiv 0$, for the solution $\bu$ of the PDE \eqref{eq:pde}. Hence, the term \emph{residual} is justified for \eqref{eq:res1}. 

The strategy of PINNs, following \cite{Lag1}, is to minimize the \emph{residual} \eqref{eq:res1}, over the admissible set of tuning parameters $\theta \in \Theta$ i.e 
\begin{equation}
    \label{eq:pinn1}
    {\rm Find}~\theta^{\ast} \in \Theta:\quad \theta^{\ast} = {\rm arg}\min\limits_{\theta \in \Theta} \|\res_{\theta}\|_{Y}.
\end{equation}
Realizing that $Y = L^p(\dom)$ for some $1 \leq p < \infty$, we can equivalently minimize,
\begin{equation}
    \label{eq:pinn2}
    {\rm Find}~\theta^{\ast} \in \Theta:\quad \theta^{\ast} = {\rm arg}\min\limits_{\theta \in \Theta} \|\res_{\theta}\|^p_{L^p(\dom)} = {\rm arg}\min\limits_{\theta \in \Theta} \int\limits_{\dom} |\res_{\theta}(y)|^p dy. 
\end{equation}
As it will not be possible to evaluate the integral in \eqref{eq:pinn2} exactly, we need to approximate it numerically by a quadrature rule. To this end, we use the quadrature rules \eqref{eq:quad} discussed earlier and select the \emph{training set} $\train = \{y_n\}$ with $y_n \in \dom$ for all $1 \leq n \leq N$ as the quadrature points for the quadrature rule \eqref{eq:quad} and consider the following \emph{loss function}:
\begin{equation}
    \label{eq:lf1}
    J(\theta):= \sum\limits_{n=1}^N w_n |\res_{\theta}(y_n)|^p = \sum\limits_{n=1}^N w_n \left| \df(\bu_{\theta}(y_n)) - \f(y_n) \right|^p.
\end{equation}
It is common in machine learning \cite{DLbook} to regularize the minimization problem for the loss function i.e we seek to find,
\begin{equation}
\label{eq:lf2}
\theta^{\ast} = {\rm arg}\min\limits_{\theta \in \Theta} \left(J(\theta) + \lambda_{reg} J_{reg}(\theta) \right).
\end{equation}  
Here, $J_{reg}:\Theta \to \R$ is a \emph{regularization} (penalization) term. A popular choice is to set  $J_{reg}(\theta) = \|\theta_W\|^q_q$ for either $q=1$ (to induce sparsity) or $q=2$. The parameter $0 \leq \lambda_{reg} \ll 1$ balances the regularization term with the actual loss $J$ \eqref{eq:lf1}. 

The proposed algorithm for computing this PINN is given below,
\begin{algorithm} 
\label{alg:PINN} {\bf Finding a physics informed neural network to approximate the solution of the very general form PDE \eqref{eq:pde}}. 
\begin{itemize}
\item [{\bf Inputs}:] Underlying domain $\dom$, differential operator $\df$ and input source term $\f$ for the PDE \eqref{eq:pde}, quadrature points and weights for the quadrature rule \eqref{eq:quad}, non-convex gradient based optimization algorithms.
\item [{\bf Goal}:] Find PINN $\bu^{\ast}= \bu_{\theta^{\ast}}$ for approximating the PDE \eqref{eq:pde}. 
\item [{\bf Step $1$}:] Choose the training set $\train = \{y_n\}$ for $y_n \in \dom$, for all $1 \leq n \leq N$ such that $\{y_n\}$ are quadrature points for the underlying quadrature rule \eqref{eq:quad}.
\item [{\bf Step $2$}:] For an initial value of the weight vector $\overline{\theta} \in \Theta$, evaluate the neural network $\bu_{\overline{\theta}}$ \eqref{eq:ann1}, the PDE residual \eqref{eq:res1}, the loss function \eqref{eq:lf2} and its gradients to initialize the underlying optimization
algorithm.
\item [{\bf Step $3$}:] Run the optimization algorithm till an approximate local minimum $\theta^{\ast}$ of \eqref{eq:lf2} is reached. The map $\bu^{\ast} = \bu_{\theta^{\ast}}$ is the desired PINN for approximating the solution $\bu$ of the PDE \eqref{eq:pde}. 
\end{itemize}
\end{algorithm}

\section{Korteweg de-Vries $\&$ Kawahara equations}
\label{sec:3}
We will apply the PINNs algorithm \ref{alg:PINN} to several examples of non-linear dispersive PDEs. We start with the well-known KdV-Kawahara equations.
\subsection{The underlying PDEs}
The general form of the KdV-Kawahara equation is given by,
\begin{equation}
    \label{eq:heat}
    \begin{aligned}
    u_t + u u_x + \alpha u_{xxx} - \beta u_{xxxxx}&= 0, \quad \forall ~ x\in (0,1), \, t \in (0,T), \\
    u(x,0) &= \bar{u}(x), \quad \forall ~ x \in (0,1), \\
    u(0,t) &= h_1(t), \quad \forall ~ t \in (0,T), \\
    u(1,t) &= h_2(t), \quad \forall ~ t \in (0,T), \\
    u_x(0,t) &=h_3(t), \quad \forall ~ t \in (0,T), \\
    u_x(1,t) &=h_4(t), \quad \forall ~ t \in (0,T), \\
    u_{xx}(1,t) &= h_5(t), \quad \forall ~ t \in (0,T). 
    \end{aligned}
\end{equation}
Here $\alpha, \beta $ are non-negative real constants.
Note that if $\beta =0$, then the above equation is called Korteweg de-Vries (KdV) equation, and if $\beta \neq 0$, then the above equation is called the Kawahara equation. It is well known that KdV equation plays a pivotal role in the modeling of shallow water waves, and in particular, the one-dimensional waves of small but finite amplitude in dispersive systems can be described by the KdV equation. However, under certain circumstances, the coefficient of the third order derivative in the KdV equation may become very small or even zero \cite{hunter}. In such a scenario, one has to take account of the higher order effect of dispersion in order to balance the nonlinear effect, which leads to the Kawahara equation. 

For the sake of simplicity it will be assumed $\alpha=\beta=1$ in the upcoming analysis, since their values are not relevant in the present setting, while emphasizing that that the subsequent analysis also holds for the case $\beta=0$ (KdV equations). Regarding the existence and stability of solutions to \eqref{eq:heat}, we closely follow the work by Faminskii $\&$ Larkin \cite{andrei}, and recall the following result.
\begin{theorem}
\label{002}
For any integer $k\ge 0$, $n \in \N$, $l=1$ or $2$, define the spaces
\begin{align*}
\mathcal{X}_k((0,1)\times (0,T))& := \Big\{  u: \partial^n_t u  \in C([0,T]; H^{5(k-n)}(0,1)) \cap L^2((0,T); H^{5(k-n)+1}(0,1))\Big\}, \\
\mathcal{B}^l_k(0,T)& := \displaystyle \prod_{j=0}^{l} H^{k + (2-j)/5} (0,T).
\end{align*}
Let $\bar{u} \in H^{5k}(0,1)$, boundary data $(h_1, h_3) \in \mathcal{B}^1_k(0,T)$, and $(h_2, h_4, h_5) \in \mathcal{B}^2_k(0,T)$ satisfy the natural compatibility conditions. Then there exists a unique solution $u \in \mathcal{X}_k$, and the flow map is Lipschitz continuous on any ball in the corresponding norm.
\end{theorem}
By choosing appropriate values of $k$ (for our purpose, we take $k=2$) in the above theorem, we readily infer the existence of classical solutions of the Kawahara equations \eqref{eq:heat} by the embedding of Sobolev spaces in the $C^{\ell}$ spaces. 
\subsection{PINNs for the KdV-Kawahara Equations \eqref{eq:heat}}
We apply algorithm \ref{alg:PINN} to approximate the solutions of \eqref{eq:heat}. To this end, we need the following steps,
\subsubsection{Training Set.}
\label{sec:train}
Let us define the space-time domain $\Omega_T = (0,1) \times (0,T)$, and divide the training set $\train = \train_{int} \cup \train_{sb} \cup \train_{tb}$ of the abstract PINNs algorithm \ref{alg:PINN} into the following three subsets, 
\begin{itemize}
\item [(a)] Interior training points $\train_{int}=\{y_n\}$ for $1 \leq n \leq N_{int}$, with each $y_n = (x_n,t_n) \in \Omega_T$. We use low-discrepancy Sobol points as training points.
\item [(b)] Spatial boundary training points $\train_{sb} = (0,t_n) \cup (1,t_n)$ for $1 \leq n \leq N_{sb}$, and the points $t_n$ chosen as low-discrepancy Sobol points.
\item [(c)] Temporal boundary training points $\train_{tb} = \{x_n\}$, with $1 \leq n \leq N_{tb}$ and each $x_n \in (0,1)$, chosen as low-discrepancy Sobol points. 
\end{itemize}
\subsubsection{Residuals}
To define residuals for the neural network $u_{\theta} \in C^5([0,T]\times [0,1])$, defined by \eqref{eq:ann1}, with $\theta \in \Theta$ as the set of tuning parameters, we use the hyperbolic tangent $\tanh$ activation function, i.e., $\sigma = \tanh$. With this setting, we define the following residuals
\begin{itemize}
\item [(a)] Interior Residual given by,
\begin{equation}
\label{eq:hres1}
\res_{int,\theta}(x,t):= \partial_t u_{\theta}(x,t) + u_{\theta} (u_{\theta})_x(x,t) + (u_{\theta})_{xxx}(x,t) - (u_{\theta})_{xxxxx}(x,t).
\end{equation}
Note that the above residual is well-defined and $\res_{int,\theta} \in C([0,T]\times [0,1])$ for every $\theta \in \Theta$. 
\item [(b)] Spatial boundary Residual given by,
\begin{equation}
\begin{aligned}
    \label{eq:hres2}
    \res_{sb1,\theta}(0,t) & := u_{\theta}(0,t) -h_1(t), \quad \forall  t \in (0,T), \\
    \res_{sb2,\theta}(1,t) & := u_{\theta}(1,t) -h_2(t), \quad \forall  t \in (0,T), \\
    \res_{sb3,\theta}(0,t) & := (u_{\theta})_x(0,t) -h_3(t), \quad \forall  t \in (0,T),\\
    \res_{sb4,\theta}(1,t) & := (u_{\theta})_x(1,t) -h_4(t), \quad \forall  t \in (0,T),\\
    \res_{sb5,\theta}(1,t) & := (u_{\theta})_{xx}(1,t) -h_5(t), \quad \forall  t \in (0,T).
    \end{aligned}
\end{equation}
Given the fact that the neural network and boundary data are smooth, above residuals are well-defined. 
\item [(c)] Temporal boundary Residual given by,
\begin{equation}
    \label{eq:hres3}
    \res_{tb,\theta}(x):= u_{\theta}(x,0) - \bar{u}(x), \quad \forall x \in (0,1). 
\end{equation}
Again the above quantity is well-defined and $\res_{tb,\theta} \in C^5((0,1))$, as both the initial data and the neural network are smooth. 
\end{itemize}
\subsubsection{Loss function}
We set the following loss function
\begin{equation}
    \label{eq:hlf}
    J(\theta):= \sum\limits_{n=1}^{N_{tb}} w^{tb}_n|\res_{tb,\theta}(x_n)|^2 + \sum\limits_{n=1}^{N_{sb}} \sum\limits_{i=1}^{5}  w^{sb}_n|\res_{sbi,\theta}(t_n)|^2 + \lambda \sum\limits_{n=1}^{N_{int}} w^{int}_n|\res_{int,\theta}(x_n,t_n)|^2 .
\end{equation}
Here the residuals are defined by \eqref{eq:hres3}, \eqref{eq:hres2}, \eqref{eq:hres1}, $w^{tb}_n$ are the $N_{tb}$ quadrature weights corresponding to the temporal boundary training points $\train_{tb}$, $w^{sb}_n$ are the $N_{sb}$ quadrature weights corresponding to the spatial boundary training points $\train_{sb}$ and $w^{int}_n$ are the $N_{int}$ quadrature weights corresponding to the interior training points $\train_{int}$. Furthermore, $\lambda$ is a hyperparameter for balancing the residuals, on account of the PDE and the initial and boundary data, respectively.
\subsection{Estimate on the generalization error}
We are interested in estimating the following generalization error for the PINN $u^* =u_{\theta^*}$ with loss function \eqref{eq:hlf}, for approximating the solution of \eqref{eq:heat}:
\begin{equation}
    \label{eq:hegen}
    \er_{G}:= \left(\int\limits_0^T \int\limits_0^1 |u(x,t) - u^{\ast}(x,t)|^2 dx dt \right)^{\frac{1}{2}}.
\end{equation}
We are going to estimate the generalization error in terms of the \emph{training error} that we define as,
\begin{equation}
    \label{eq:hetrain}
    \er^2_{T}:= \underbrace{\sum\limits_{n=1}^{N_{tb}} w^{tb}_n|\res_{tb,\theta^{\ast}}(x_n)|^2}_{(\er_T^{tb})^2} + \underbrace{\sum\limits_{n=1}^{N_{sb}} \sum\limits_{i=1}^{5} w^{sb}_n|\res_{sbi,\theta^{\ast}}(t_n)|^2}_{(\er_T^{sb})^2} +  \lambda\underbrace{\sum\limits_{n=1}^{N_{int}} w^{int}_n|\res_{int,\theta^{\ast}}(x_n,t_n)|^2}_{(\er_T^{int})^2}.
\end{equation}
Note that the training error can be readily computed \emph{a posteriori} from the loss function \eqref{eq:hlf}. 

We also need the following assumptions on the quadrature error. For any function $g \in C^k(\Omega)$, the quadrature rule corresponding to quadrature weights $w^{tb}_n$ at points $x_n \in \train_{tb}$, with $1 \leq n \leq N_{tb}$, satisfies 
\begin{equation}
    \label{eq:hquad1}
    \left| \int\limits_{\Omega} g(x) dx - \sum\limits_{n=1}^{N_{tb}} w^{tb}_n g(x_n)\right| \leq C^{tb}_{quad}(\|g\|_{C^k}) N_{tb}^{-\alpha_{tb}}.
\end{equation}
For any function $g \in C^k(\partial\Omega \times [0,T])$, the quadrature rule corresponding to quadrature weights $w^{sb}_n$ at points $(x_n,t_n) \in \train_{sb}$, with $1 \leq n \leq N_{sb}$, satisfies 
\begin{equation}
    \label{eq:hquad2}
    \left| \int\limits_0^T \int\limits_{\partial\Omega} g(x,t) ds(x) dt - \sum\limits_{n=1}^{N_{sb}} w^{sb}_n g(x_n,t_n)\right| \leq C^{sb}_{quad}(\|g\|_{C^k}) N_{sb}^{-\alpha_{sb}}.
\end{equation}
Finally, for any function $g \in C^\ell(\Omega \times [0,T])$, the quadrature rule corresponding to quadrature weights $w^{int}_n$ at points $(x_n,t_n) \in \train_{int}$, with $1 \leq n \leq N_{int}$, satisfies 
\begin{equation}
    \label{eq:hquad3}
    \left| \int\limits_0^T \int\limits_{\Omega} g(x,t) dx dt - \sum\limits_{n=1}^{N_{int}} w^{int}_n g(x_n,t_n)\right| \leq C^{int}_{quad}(\|g\|_{C^\ell}) N_{int}^{-\alpha_{int}}.
\end{equation}
In the above, $\alpha_{int},\alpha_{sb},\alpha_{tb} > 0$ and in principle, different order quadrature rules can be used. We estimate the generalization error for the PINN in the following,
\begin{theorem}
\label{thm:heat}
Let $u \in C^5([0,1] \times [0,T])$ be the unique classical solution of the Korteweg de-Vries $\&$ Kawahara equation \eqref{eq:heat}. Let $u^{\ast} = u_{\theta^{\ast}}$ be a PINN generated by algorithm \ref{alg:PINN}, corresponding to loss function \eqref{eq:lf2}, \eqref{eq:hlf}. Then the generalization error \eqref{eq:hegen} can be estimated as, 
\begin{equation}
\label{result_01}
	\begin{aligned}
	\epsilon_G &\leq C_1\big(\epsilon_T^{tb} + \epsilon_T^{int} + C_2(\epsilon_T^{sb}) + C_3(\epsilon_T^{sb})^{1/2} \\
	&+ (C_{quad}^{tb})^{1/2} N_{tb}^{-\alpha_{tb} / 2} + (C_{quad}^{int})^{1/2} N_{int}^{-\alpha_{int} / 2} + C_2(C_{quad}^{sb})^{1/2} N_{sb}^{-\alpha_{sb} / 2} + C_3(C_{quad}^{sb})^{1/4} N_{sb}^{-\alpha_{sb} / 4}\big),
	\end{aligned}
	\end{equation}
	where 
\begin{equation}
\begin{aligned}
C_1 &= \sqrt{T + 2C_4T^2e^{2C_4T}}, \quad
C_2 = \sqrt{\Vert u \Vert_{C_t^0C_x^0} + 1}, \\
C_3 &= \sqrt{10(\Vert u^* \Vert_{C_t^0C^4_x} + \Vert u \Vert_{C_t^0C^4_x})T^{1/2}}, \quad
C_4 = \Vert u^* \Vert_{C_t^0C_x^1} + \frac{1}{2}\Vert u \Vert_{C_t^0C_x^1} + \frac{1}{2},
\end{aligned}
\end{equation}
and $C_{quad}^{tb} = C_{quad}^{tb}(\|\res_{tb,\theta^{\ast}}\|_{C^5})$, $C_{quad}^{sb} = C_{quad}^{sb}(\sum\limits_{i=1}^{5}\|\res_{sbi,\theta^{\ast}}\|_{C^{3}})$ and $C_{quad}^{int} = C_{quad}^{int}(\|\res_{int,\theta^{\ast}}\|_{C^{0}})$ are the constants defined by the quadrature error \eqref{eq:hquad1}, \eqref{eq:hquad2}, \eqref{eq:hquad3}, respectively. 
\end{theorem}
\begin{proof}
It is easy to see that the error $\hat{u}: u^{\ast} - u$ satisfies the following equations,
\begin{equation}
    \label{eq:herr}
\begin{aligned}
    \hat{u}_t +  \hat{u}_{xxx} -  \hat{u}_{xxxxx}+ u^{\ast} u^{\ast}_x - uu_x&=  \res_{int}, \quad x\in (0,1)~ t \in (0,T), \\
    \hat{u}(x,0) &= \res_{tb}(x), \quad  x \in (0,1), \\
    \hat{u}(0,t) &= \res_{sb1}(0,t), \quad  t \in (0,T), \\
    \hat{u}(1,t) &= \res_{sb2}(1,t), \quad  t \in (0,T),\\
    \hat{u}_x(0,t) &= \res_{sb3}(0,t), \quad  t \in (0,T), \\
    \hat{u}_x(1,t) &= \res_{sb4}(1,t), \quad  t \in (0,T), \\
    \hat{u}_{xx}(1,t) &= \res_{sb5}(1,t), \quad  t \in (0,T).
    \end{aligned}    
\end{equation}
Here, we have denoted $\res_{int} = \res_{int,\theta^{\ast}}$ for notational convenience and analogously for the residuals $\res_{tb},\res_{sb}.$
Note that
$$
 u^{\ast} u^{\ast}_x - uu_x =  \hat{u}  \hat{u}_x  + u  \hat{u}_x +  \hat{u} u_x;  \,\, \hat{u}  \hat{u}_{xxx} = ( \hat{u} \hat{u}_{xx})_x - \frac12 ( \hat{u}^2_x)_x, \,\, \hat{u}  \hat{u}_{xxxxx} = ( \hat{u} \hat{u}_{xxxx})_x - ( \hat{u}_x \hat{u}_{xxx})_x+ \frac12 ( \hat{u}^2_{xx})_x.
$$
Multiplying both sides of the PDE \eqref{eq:herr} with $\hat{u}$, integrating over the domain and afterwards by parts yields,
\begin{equation}
\begin{split}
\frac{1}{2}\frac{d}{dt}\int_0^1 \hat{u}^2\,dx &= - \int_0^1 \hat{u}\hat{u}_{xxx}\,dx + \int_0^1 \hat{u}\hat{u}_{xxxxx}\,dx - \int_0^1\hat{u}(\hat{u}\hat{u}_x - u\hat{u}_x + u_x\hat{u})\,dx + \int_0^1 \hat{u}\res_{int}\,dx \\
&\leq -\left. \hat{u}_{xx}\hat{u} \right\vert_0^1 + \frac{1}{2} \left. (\hat{u}_x)^2 \right\vert_1 + \left. \hat{u}_{xxxx}\hat{u} \right\vert_0^1 - \left. \hat{u}_{xxx}\hat{u}_x \right\vert_0^1 + \frac{1}{2} \left. (\hat{u}_{xx})^2 \right\vert_1 \\
&- \int_0^1 \hat{u}^2\hat{u}_x\,dx - (\left. \frac{1}{2}\hat{u}^2u \right\vert_0^1 - \frac{1}{2}\int_0^1 \hat{u}^2u_x\,dx) - \int_0^1 \hat{u}^2u_x\,dx + \int_0^1 \hat{u}\res_{int}\,dx \\
&\leq \Vert \hat{u} \Vert_{C^4_x}(|\res_{sb1}| + |\res_{sb2}| + |\res_{sb3}| + |\res_{sb4}|) + \frac{1}{2} (\res_{sb3}^2 + \res_{sb5}^2) \\
&+ (\Vert u^* \Vert_{C_x^1} + \frac{1}{2}\Vert u \Vert_{C_x^1})  \int_0^1 \hat{u}^2\,dx + \frac{1}{2}\Vert u \Vert_{C_x^0}(\res_{sb1}^2 + \res_{sb2}^2) + \frac{1}{2}\int_0^1 \res_{int}^2 \,dx + \frac{1}{2}\int_0^1 \hat{u}^2 \,dx \\
&\leq (\Vert u^* \Vert_{C_t^0C^4_x} + \Vert u \Vert_{C_t^0C^4_x})(|\res_{sb1}| + |\res_{sb2}| + |\res_{sb3}| + |\res_{sb4}|) \\
&+ \frac{1}{2} (\res_{sb3}^2 + \res_{sb5}^2) + \frac{1}{2}\Vert u \Vert_{C_t^0C_x^0}(\res_{sb1}^2 + \res_{sb2}^2) + \frac{1}{2}\int_0^1 \res_{int}^2 \,dx \\
&+ (\Vert u^* \Vert_{C_t^0C_x^1} + \frac{1}{2}\Vert u \Vert_{C_t^0C_x^1} + \frac{1}{2})  \int_0^1 \hat{u}^2\,dx \\
&=: C_1(\sum\limits_{i=1}^5 |\res_{sbi}|) + C_2(\sum\limits_{i=1}^5 \res_{sbi}^2) + \frac{1}{2}\int_0^1 \res_{int}^2 \,dx + C_3\int_0^1 \hat{u}^2\,dx.
\end{split}
\end{equation}
Here the mixed norm is defined as $\Vert u \Vert_{C_t^mC_x^n} := \sum\limits_{0 \leq i \leq m, 0 \leq j \leq n}\Vert \frac{\partial^i}{\partial t^i}\frac{\partial^j}{\partial x^j} u \Vert_{C^0_{x, t}}$. Then integrating the above inequality over $[0,\bar{T}]$ for any $\bar{T} \leq T$, using Cauchy-Schwarz and Gronwall's inequalities, we obtain
	\begin{equation}
	\begin{aligned}
	&\int_0^1 \hat{u}(x, \bar{T})^2\,dx \\
	&\leq \int_0^1 \res_{tb}^2\,dx + 2C_1T^{1/2}\sum\limits_{i=1}^5(\int_0^T\res_{sbi}^2\,dt)^{1/2} + 2C_2\sum\limits_{i=1}^5(\int_0^T\res_{sbi}^2\,dt)  
	+ \int_0^T\int_0^1 \res_{int}^2 \,dxdt + 2C_3\int_0^{\bar{T}}\int_0^1 \hat{u}^2\,dxdt \\
	&\leq (1 + 2C_3Te^{2C_3T})  \big(\int_0^1 \res_{tb}^2\,dx + 10C_1T^{1/2}(\sum\limits_{i=1}^5\int_0^T \res_{sbi}^2\,dt)^{1/2} 
	+ 2C_2\sum\limits_{i=1}^5(\int_0^T\res_{sbi}^2\,dt) + \int_0^T\int_0^1 \res_{int}^2 \,dxdt \big) .
	\end{aligned}
	\label{eq:kawa_hat_eq3}
	\end{equation}
A further integration of \eqref{eq:kawa_hat_eq3} with respect to $\bar{T}$ results in
\begin{equation}
\begin{aligned}
\epsilon_G^2 := \int_0^T\int_0^1 \hat{u}(x, \bar{T})^2\,dxd\bar{T}
&\leq (T + 2C_3T^2e^{2C_3T})  \big(\int_0^1 \res_{tb}^2\,dx + 10C_1T^{1/2}(\sum\limits_{i=1}^5\int_0^T \res_{sbi}^2\,dt)^{1/2} \\
&\quad + 2C_2\sum\limits_{i=1}^5(\int_0^T\res_{sbi}^2\,dt) + \int_0^T\int_0^1 \res_{int}^2 \,dxdt \big),
\end{aligned}
\label{eq:kawa_hat_eq4}
\end{equation}
	with
	\begin{equation}
	\begin{aligned}
	C_1 =  \Vert u \Vert_{C_t^0C^4_x} + \Vert u^* \Vert_{C_t^0C^4_x}, \quad 
	C_2 = \frac{1}{2}\Vert u \Vert_{C_t^0C_x^0} + \frac{1}{2}, \quad
	C_3 = \Vert u^* \Vert_{C_t^0C_x^1} + \frac{1}{2}\Vert u \Vert_{C_t^0C_x^1} + \frac{1}{2}.
	\end{aligned}
\end{equation}
Eventually, applying the estimates \eqref{eq:hquad1}, \eqref{eq:hquad2}, \eqref{eq:hquad3} on the quadrature error, and definition of training errors \eqref{eq:hetrain}, yields the desired inequality \eqref{result_01}.
\end{proof}
\subsection{Numerical experiments}
\subsubsection{Implementation}
The PINNs algorithm \ref{alg:PINN} has been implemented within the PyTorch framework \cite{Pas} and the code can be downloaded from \textbf{https://github.com/baigm11/DispersivePinns}. As is well documented \cite{KAR1, KAR2, MM1}, the coding and implementation of PINNs is extremely simple. Only a few lines of Python code suffice for this purpose. All the numerical experiments were performed on a single GeForce GTX1080 GPU.

The PINNs algorithm has the following hyperparameters, the number of hidden layers $K-1$, the width of each hidden layer $d_k \coloneqq \bar{d}$ in \eqref{eq:ann1}, the specific activation function $A$, the parameter $\lambda$ in the loss function \eqref{eq:hlf}, the regularization parameter $\lambda_{reg}$ in the cumulative loss function \eqref{eq:lf2} and the specific gradient descent algorithm for approximating the optimization problem \eqref{eq:lf2}. We use the hyperbolic tangent $\tanh$ activation function, thus ensuring that all the smoothness hypothesis on the resulting neural networks, as required in all bounds on generalization error below, are satisfied. Moreover, we use the second-order LBFGS method \cite{Fle} as the optimizer. We follow the ensemble training procedure of \cite{LMR1} in order to choose the remaining hyperparameters. To this end, we consider a range of values, shown in Table \ref{tab:1}, for the number of hidden layers, the depth of each hidden layer, the parameter $\lambda$ and the regularization parameter $\lambda_{reg}$. For each configuration in the ensemble, the resulting model is retrained (in parallel) $n_\theta$ times with different random starting values of the trainable weights in the optimization algorithm and the one yielding the smallest value of the training loss is selected.
\begin{table}[htbp] 
    \centering
    \renewcommand{\arraystretch}{1.1} 
    
    \footnotesize{
        \begin{tabular}{c c c c c c c } 
            \toprule
            \bfseries   &\bfseries $K-1$  & \bfseries $d$  &\bfseries $q$ & \bfseries $\lambda_{reg}$  &\bfseries $\lambda$ &\bfseries $n_\theta$ \\ 
            \midrule
            \midrule
             KdV Equation & 4, 8  & 20, 24, 28 &2& 0& 0.1, 1, 10 & 5\\
            \midrule
             Kawahara Equation & 4, 8, 12  & 20, 24, 28, 32 &2& 0& 0.1, 1, 10 & 5\\
            \midrule
            CH Equation     & 4, 8  & 20, 24, 28 &2& 0& 0.1, 1, 10 & 5\\
            \midrule
             BO Equation, Single Soliton & 4, 8, 12  & 20, 24, 28, 32 &2& 0& 0.1, 1, 10 & 5\\
                \midrule 
             BO Equations, Double Soliton & 4, 8  & 20, 24, 28 &2& 0& 0.1, 1, 10 & 5\\

            \bottomrule
        \end{tabular}
    \caption{Hyperparameter configurations employed in the ensemble training of PINNs.}
        \label{tab:1}
    }
\end{table}
\subsubsection{KdV equation}
We set $\beta = 0$ in \eqref{eq:heat} to recover the KdV equation and consider the well-known numerical benchmarks of single and double soliton solutions, with exact solution formulas for both cases. 

For the single soliton, the exact solution is given by,
\begin{equation}
	u(x, t) = 9\text{sech}^2(\sqrt{3/4}(x - 3t)),
	\label{eq:kdv_single_exact}
\end{equation}
representing a single bump moving to the right with speed 3 with initial peak at $x = 0$. 
\begin{figure}[h!]
    \begin{subfigure}{.49\textwidth}
        \centering
        \includegraphics[width=1\linewidth]{{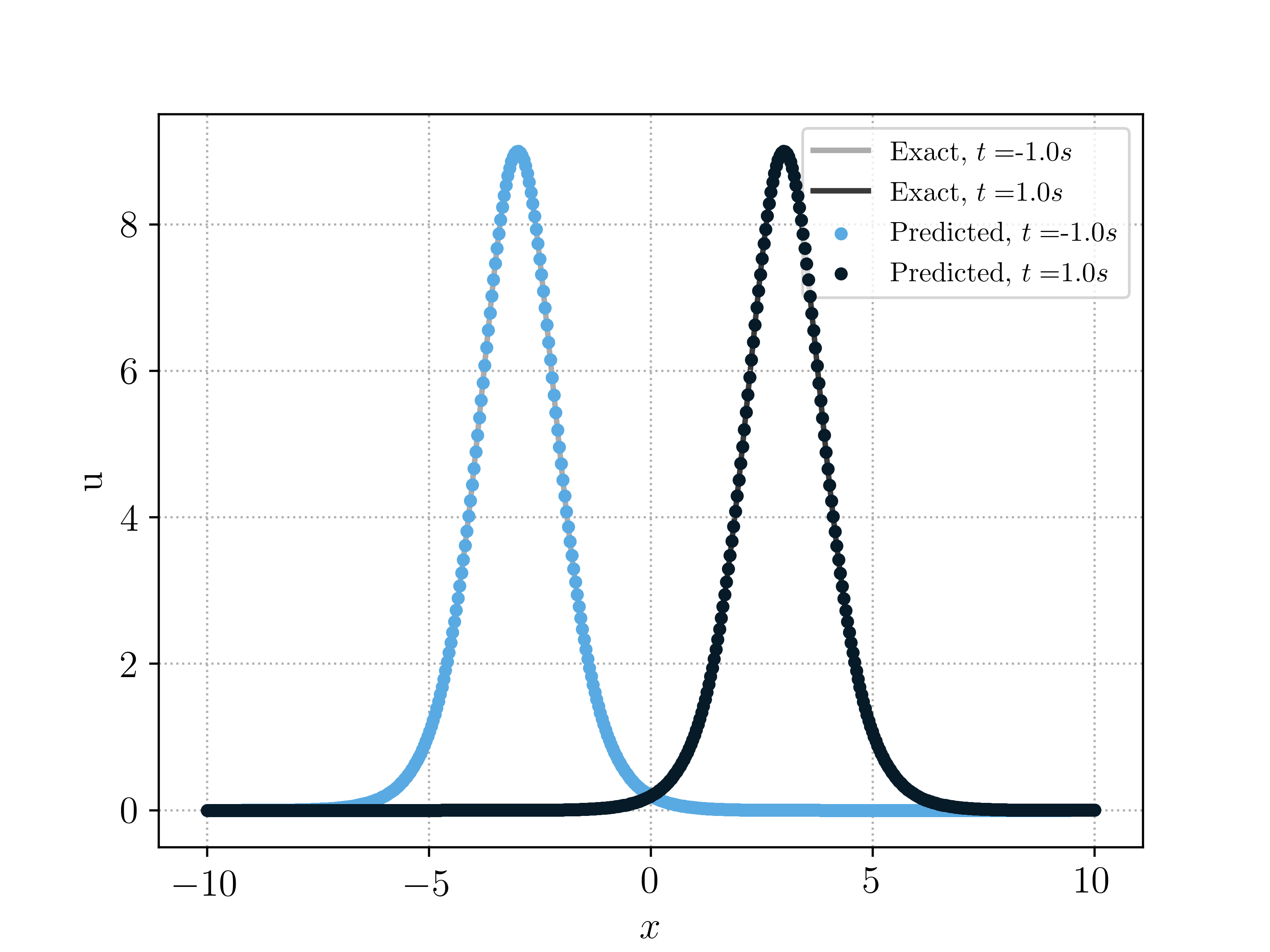}}
        \caption{Single soliton}
    \end{subfigure}
    \begin{subfigure}{.49\textwidth}
        \centering\
        \includegraphics[width=1\linewidth]{{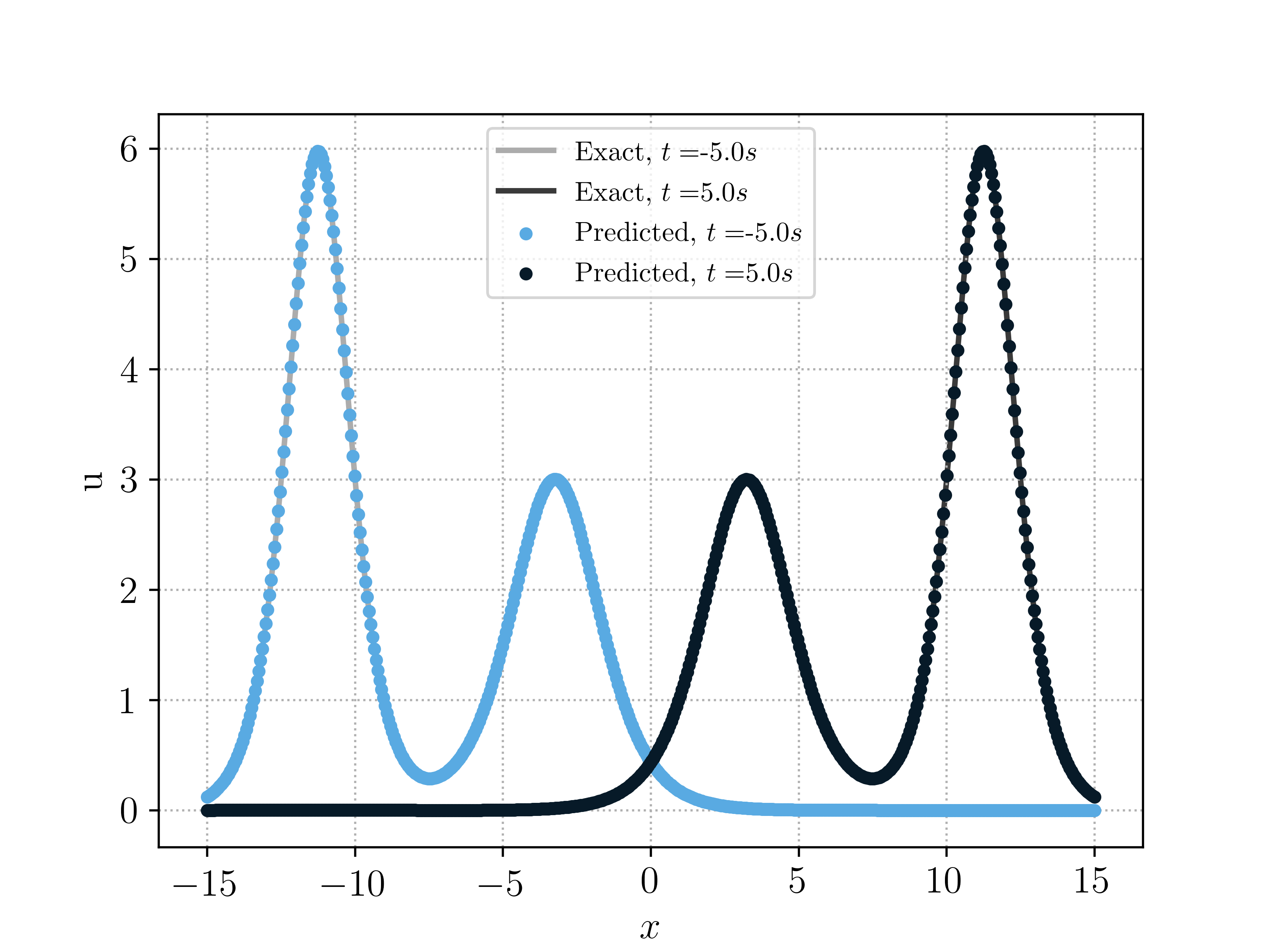}}
        \caption{Double soliton}
    \end{subfigure}
    \caption{The exact and PINN solutions of single and double soliton test case of KdV equation.}
    \label{fig:kdv}
\end{figure}

The ensemble training for the PINNs in this case resulted in the selection of hyperparameters, reported in Table \ref{tab:kdv}. We plot the exact solution and the approximate solution, computed with the PINNs algorithm \ref{alg:PINN} in figure \ref{fig:kdv} (left). As seen from this figure, PINNs provide a very accurate approximation for the single soliton. This is further verified in the extremely low generalization errors reported in Table \ref{tab:kdv}, showcasing the ability of PINNs to accurately approximate single solitons for the KdV equation.

\begin{table}[htbp] 
    \centering
    \renewcommand{\arraystretch}{1.1} 
    
    \footnotesize{
        \begin{tabular}{ c c c c c c c c c c} 
            \toprule
            \bfseries   &\bfseries $N_{int}$  & \bfseries $N_{sb}$& \bfseries $N_{tb}$  &\bfseries $K-1$ & \bfseries $d$   &\bfseries $\lambda$ &\bfseries $\er_T$&\bfseries $\er_G^r$ \\ 
            \midrule
            \midrule
            Single Soliton &2048   & 512 & 512 &4&20& 0.1 &0.000236& 0.00338\% \\
            \midrule 
            Double Soliton           &4096   & 1024 & 1024 &4&32& 1 &0.000713& 0.059\% \\

            \bottomrule
        \end{tabular}
    \caption{Best performing \textit{Neural Network} configurations for the single soliton and double soliton problem. Low-discrepancy Sobol points are used for every reported numerical example.}
    \label{tab:kdv}
    }
\end{table}

For the double soliton, the exact solution is given by,
\begin{equation}
u(x, t) = 6(b - a)\frac{b\textrm{csch}^2(\sqrt{b/2}(x - 2bt)) + a\textrm{sech}^2(\sqrt{a/2}(x - 2at))}{\big(\sqrt{a}\tan(\sqrt{a/2}(x - 2at)) - \sqrt{b}\tanh(\sqrt{b/2}(x - 2bt))\big)^2},
\label{eq:kdv_d}
\end{equation}
for any real numbers a and b where we have used $a = 0.5$ and $b = 1$ in the numerical experiment. \eqref{eq:kdv_d} represents two solitary waves that “collide” at $t = 0$ and separate for $t > 0$. For large $|t|$, $u(\cdot, t)$ is close to a sum of two solitary waves at different locations.

The ensemble training for the PINNs in this case resulted in the selection of hyperparameters, reported in Table \ref{tab:kdv} (bottom row). We plot the exact solution and the approximate solution, computed with the PINNs algorithm \ref{alg:PINN} in figure \ref{fig:kdv} (right). As seen from this figure, PINNs provide a very accurate approximation for the double soliton, which is further verified in the extremely low generalization errors reported in Table \ref{tab:kdv}. Thus, PINNs are able to approximate KdV solitons to very high accuracy. 

\begin{table}[htbp] 
    \centering
    \renewcommand{\arraystretch}{1.1} 
    
    \footnotesize{
        \begin{tabular}{ c c c c } 
            \toprule
            \bfseries max\_iters  &\bfseries training time$/s$ &\bfseries $\eps_T$ &\bfseries $\eps_G^r$ \\ 
            \midrule
            \midrule
            100  & 4   & 6.75e-02 &  1.84e-01  \\
            \midrule 
              500  & 21  & 2.41e-03   &  1.65e-03  \\
                \midrule 
              1000  & 44  & 7.34e-04  &   4.92e-04  \\
                \midrule 
              2000  & 61  & 2.36e-04  &  3.38e-05  \\
              
            \bottomrule
        \end{tabular}
    \caption{Results of different training iterations for single soliton case of KdV equation.}
		\label{tab:kdv_single_cp}
    }
\end{table}

Lastly, it is natural to investigate the computational cost of the PINNs in approximating the KdV solutions. The computational cost is dominated by the training i.e, the number of LBFGS iterations that are needed to minimize the training error. We report the training times (in seconds) for different number of iterations $max_{iters}$ for the single soliton test case in Table \ref{tab:kdv_single_cp} and for the double soliton test case in Table \ref{tab:kdv_double_cp}. From Table \ref{tab:kdv_single_cp}, we observe that the PINN for approximating single soliton is very fast to train, with a relative error of $1\%$ already reached with less than $500$ LBFGS iterations and a training time of approximately $20$ seconds. On the other hand, the PINN for the double soliton takes longer to train and attains an error of less than $1\%$, only with $2000$ iterations and a training time of less than $3$ minutes. This is not surprising as the double soliton has a significantly more complicated structure. Nevertheless, the overall cost is still very low, given the high-accuracy.

\begin{table}[htbp] 
    \centering
    \renewcommand{\arraystretch}{1.1} 
    
    \footnotesize{
        \begin{tabular}{ c c c c } 
            \toprule
            \bfseries max\_iters  &\bfseries training time$/s$ &\bfseries $\eps_T$ &\bfseries $\eps_G^r$ \\ 
            \midrule
            \midrule
         100   &     9  &   1.21e-01  &   4.82e-01 \\
         \midrule
         500   &    48  &   2.60e-02  &   1.30e-01 \\
         \midrule
        1000   &    95  &   7.00e-03  &   4.32e-02 \\
        \midrule
        2000   &   159  &   2.54e-03  &   1.11e-02 \\
        \midrule
        5000   &   436  &   7.89e-04   &  6.50e-04 \\
        \midrule
       10000   &   499  &   7.13e-04   &  5.88e-04 \\
              
            \bottomrule
        \end{tabular}
    \caption{Results of different training iterations for double soliton case of KdV equation.}
		\label{tab:kdv_double_cp}
    }
\end{table}
\subsubsection{Uncertainty Quantification.}
A natural advantage of PINNs lies in their ability to approximate statistical quantities for PDEs, for instance in the context of Uncertainty quantification (UQ) efficiently. 

To see this, we consider the following \emph{parameterized} initial-value problem for the KdV equations,
\begin{equation}
\begin{aligned}
    u_t + \gamma uu_x + \kappa u_{xxx} &= 0, \\
    u_0(x,\alpha,\beta,\gamma) &= \frac{\beta}{\gamma} + \frac{\alpha - \beta}{\gamma}\text{sech}^2\Big(\sqrt{\frac{\alpha - \beta}{12\kappa}}(x)\Big)
    \end{aligned}
    \label{eq:kdv_param}
\end{equation}
Here $\alpha,\beta,\gamma$ are scalar parameters that specify the initial location and amplitude for the \emph{soliton} initial data and $\kappa$ is a scalar parameter that measures the dispersivity of the medium. 

It turns out that this parametrized KdV equation \eqref{eq:kdv_param} admits an exact Soliton solution given by $u = u(x, t, \alpha, \beta, \gamma, \kappa): \Omega \subseteq \R^6 \rightarrow \R$, satisfying,
\begin{equation}
    u = \frac{\beta}{\gamma} + \frac{\alpha - \beta}{\gamma}\text{sech}^2\Big(\sqrt{\frac{\alpha - \beta}{12\kappa}}(x - (\beta + \frac{\alpha - \beta}{3})t)\Big)
    \label{eq:kdv_param_exact}
\end{equation}

We can readily see that the KdV single soliton solution \eqref{eq:kdv_single_exact} is recovered by setting $(\alpha, \beta, \gamma, \kappa) = (9, 0, 1, 1)$, solution \eqref{eq:kdv_param_exact} reduces to KdV single soliton solution \eqref{eq:kdv_single_exact}. 

We define an UQ problem by a stochastic perturbation of the initial soliton corresponding to the above choice of parameters. In particular, we choose $\alpha \sim \mathcal{U}(8.7, 9.3), \beta \sim \mathcal{U}(-0.4, 0.4), \gamma \sim \mathcal{U}(0.9, 1.1), \kappa \sim \mathcal{U}(0.9, 1.1)$ such that $\mathop{\mathbb{E}}(\alpha, \beta, \gamma, \kappa) = (9, 0, 1, 1)$. 

The UQ problem is approximated with PINNs by collocating the PINN residual resulting from \eqref{eq:kdv_param} on Sobol points, from the underlying $6$-dimensional domain. The initial and periodic boundary residuals are analogously computed. In figure \ref{fig:kdv_uq}, we plot the mean $\pm$ standard deviation, for both the initial data as well the uncertain solution at a later time and compare it with the exact solution, computed from \eqref{eq:kdv_param_exact}. We observe from this figure that the statistical quantities, computed with PINN, approximate the exact solution quite well. This qualitative observation is reinforced in the quantitative results presented in Table \ref{tab:kdv_uq}, where the generalization error, defined completely analogously to \eqref{eq:hegen} by integrating over the parameter space, is observed to be less than $0.5\%$ in approximately $5$ minutes of training time. This result highlights the ability of PINNs to approximate \emph{high-dimensional} parametric dispersive PDEs to high accuracy. 

\begin{figure}[h!]
\centering
    \includegraphics[width=8cm]{{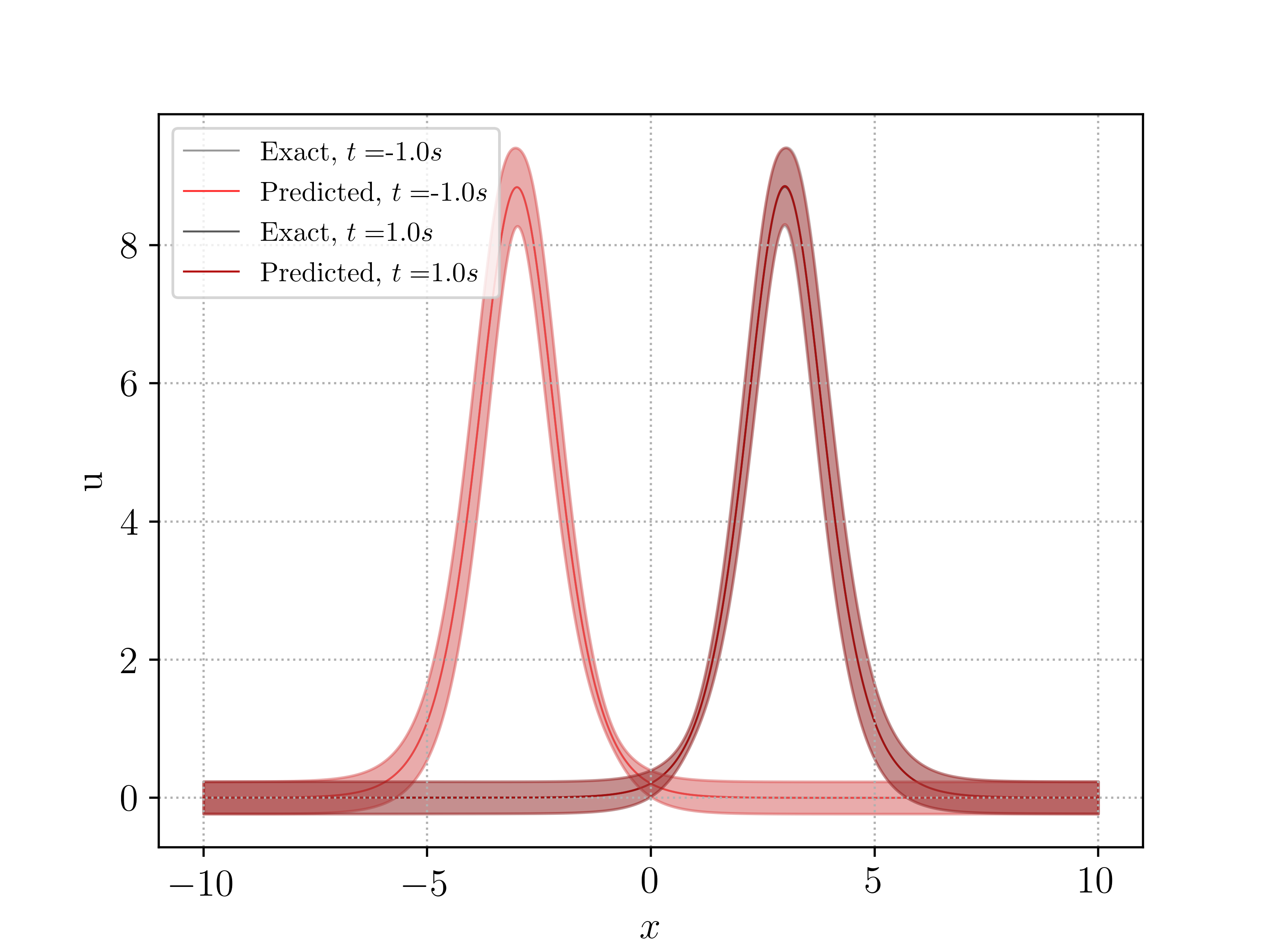}}
    \caption{The mean$\pm$std plot of exact and PINN solution of parametrized single soliton test case of parametrized KdV equation \eqref{eq:kdv_param}.}
    \label{fig:kdv_uq}
\end{figure}

\begin{table}[htbp] 
    \centering
    \renewcommand{\arraystretch}{1.1} 
    
    \footnotesize{
        \begin{tabular}{ c c c c c c c c c c} 
            \toprule
            \bfseries   &\bfseries $N_{int}$  & \bfseries $N_{sb}$& \bfseries $N_{tb}$  &\bfseries $K-1$ & \bfseries $d$  &\bfseries $\lambda$ &\bfseries $\er_T$&\bfseries $\er_G^r$ \\ 
            \midrule
            \midrule
            Single Soliton UQ &16384   & 4096 & 4096 &4&24& 0.1 &0.00351& 0.442\% \\
            \bottomrule
        \end{tabular}
    \caption{Best performing \textit{Neural Network} configurations for the single soliton UQ test case for the parametrized KdV equations \eqref{eq:kdv_param}. Low-discrepancy Sobol points are used for every reported numerical example.}
        \label{tab:kdv_uq}
    }
\end{table}

\subsubsection{Kawahara equation}
Following \cite{Car, koley1, koley2}, we consider a Kawahara-type equation which differs from Kawahara equation \eqref{eq:heat} in a first-order term $u_x$,
\begin{equation}
    u_t + u_x + uu_x + u_{xxx} -u_{xxxxx} = 0.
    \label{eq:kawa_mod}
\end{equation}
This first-order term $u_x$ is a linear perturbation and we can easily derive a similar a posteriori bound on generalization error, as for \eqref{eq:heat}. As no exact solution formulas for the double soliton test case are known for the Kawahara equation \eqref{eq:kawa_mod}, we focus on the single soliton case, with exact solutions given by \begin{equation}
\label{eq:kdvs}
u(x, t) = \frac{105}{169}\textrm{sech}^4\Big(\frac{1}{2\sqrt{13}}(x - \frac{205}{169}t - x_0)\Big).
\end{equation}
This represents a single bump moving to the right with speed $\frac{205}{169}$ with initial peak at $x = x_0$. The ensemble training selected PINNs with hyperparameters, given in Table \ref{tab:Kawa}. The resulting PINN approximation, together with the exact solution is plotted in figure \ref{fig:Kawa} and shows that the trained PINN approximates the exact solution with very high accuracy. This is further verified in the extremely low generalization error of $0.1\%$, reported in Table \ref{tab:Kawa}. 
\begin{table}[htbp] 
    \centering
    \renewcommand{\arraystretch}{1.1} 
    
    \footnotesize{
        \begin{tabular}{ c c c c c c c c c c} 
            \toprule
            \bfseries   &\bfseries $N_{int}$  & \bfseries $N_{sb}$& \bfseries $N_{tb}$  &\bfseries $K-1$ & \bfseries $d$  &\bfseries $\lambda$ &\bfseries $\er_T$&\bfseries $\er_G^r$ \\ 
            \midrule
            \midrule
            Single Soliton &2048   & 512 & 512 &4&24& 10 &0.000321& 0.101\% \\
            \bottomrule
        \end{tabular}
    \caption{Best performing \textit{Neural Network} configurations for the single soliton test case for the Kawahara equations \eqref{eq:kawa_mod}. Low-discrepancy Sobol points are used for every reported numerical example.}
        \label{tab:Kawa}
    }
\end{table}
In Table \ref{tab:kawa_single_cp}, we present training times (in seconds) for the PINNs algorithm for the Kawahara equation \ref{eq:kawa_mod}. We observe from this Table that an error of less than $1 \%$ percent is achieved in approximately $6-7$ minutes. Given the fact that the Kawahara equation requires the evaluation of $5$-th order derivatives, it is expected that each training iteration is significantly more expensive than that of the KdV equation. Table \ref{tab:kawa_single_cp} shows that this is indeed the case and partly explains the higher computational cost for the PINN to approximate the Kawahara equation. Nevertheless, the total cost is still considerably smaller than those reported for the finite difference scheme in \cite{koley1,koley2}. As an examples, to achieve $1 \%$ error, it takes approximately $15-18$ minutes for the dissipative finite-difference scheme presented in \cite{koley2}.

\begin{figure}[h!]
\centering
    \includegraphics[width=8cm]{{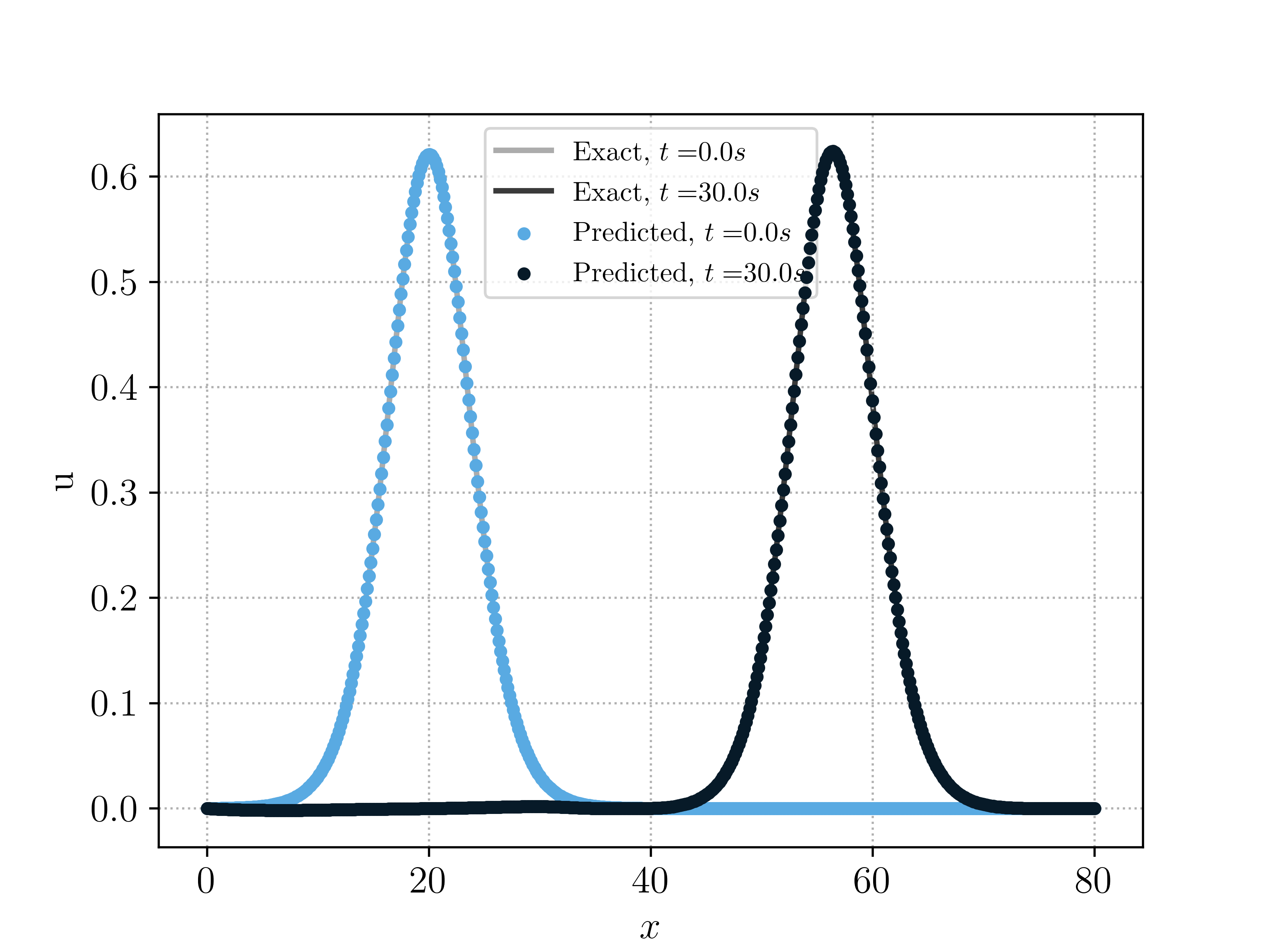}}
    \caption{The exact and PINN solution of single soliton test case of Kawahara equation \eqref{eq:kawa_mod}.}
    \label{fig:Kawa}
\end{figure}

\begin{table}[htbp] 
    \centering
    \renewcommand{\arraystretch}{1.1} 
    
    \footnotesize{
        \begin{tabular}{ c c c c } 
            \toprule
            \bfseries max\_iters  &\bfseries training time$/s$ &\bfseries $\eps_T$ &\bfseries $\eps_G^r$ \\ 
            \midrule
            \midrule
         100    &   25  &   8.89e-02   &  9.70e-01 \\
         \midrule
         500   &   127   &  4.76e-02   &  7.86e-01 \\
         \midrule
        1000   &   249   &  8.40e-03  &   1.89e-01 \\
        \midrule
        2000   &   466  &   1.06e-03   &  5.88e-03 \\
        \midrule
        5000   &   964  &   3.21e-04   &  1.01e-03 \\
              
            \bottomrule
        \end{tabular}
    \caption{Results of different training iterations for single soliton case of Kawahara equation.}
		\label{tab:kawa_single_cp}
    }
\end{table}

\section{Camassa-Holm equation}
\label{sec:4}
\subsection{The underlying PDE}
In this section, we consider the following initial-boundary value problem for the one-dimensional Camassa-Holm equation on a compact interval
\begin{equation}
\label{eq:vscl}
\begin{aligned}
u_t - u_{txx}+ 3uu_x + 2 \kappa u_x&= 2 u_x u_{xx} + u u_{xxx}, \quad \forall x\in (0,1),~t \in [0,T], \\
u(x,0) &= u_0(x), \quad \forall x \in (0,1), \\
u(0,t) &= u_{xx}(0,t) = u(1,t) = u_{xx}(1,t)=0, \quad \forall t \in [0,T].
\end{aligned}
\end{equation}
Here, $\kappa$ is a real constant. This equation models the unidirectional propagation 
of shallow water waves over a flat bottom, with $u$ representing the fluid velocity.
A key feature of the above equation is that it is completely integrable for all values of $\kappa$. A special case of \eqref{eq:vscl}, corresponding to $\kappa=0$, plays an important role in the modeling of nonlinear dispersive waves in hyperelastic rods \cite{camassa}. Regarding the existence of solutions, we report the following result which is a slight modification of the results of \cite{kwek},
\begin{theorem}
\label{thm:CH}
Let $\mathcal{X}:= \lbrace u \in H^4(0,1): u(0)=u_{xx}(0)=u(1)=u_{xx}(1)=0 \rbrace$. Then for every $u_0 \in \mathcal{X}$, the problem \eqref{eq:vscl} has
a unique solution  
$$
u \in C([0,T); \mathcal{X}) \cap C^1([0,T); H^1_0(0,1)),
$$
for some $T>0$. In addition, $u_{xx} \in C^1([0,T); H^1_0(0,1))$, and $u$ depends continuously on $u_0$ in the $H^4$-norm.
\end{theorem}
\subsection{PINNs}
To specify the PINNs algorithm \ref{alg:PINN} in this case, we start by choosing the training set, exactly as in section \ref{sec:train}. The following residuals are chosen,
\begin{itemize}
	\item Interior Residual given by,
	\begin{equation}
	\begin{aligned}
	\label{eq:hres11}
	\res_{int,\theta}(x,t)& := \partial_t u_{\theta}(x,t) - \partial_{txx} u_{\theta}(x,t) +3 u_{\theta}(x,t) (u_{\theta})_x(x,t) + 2 \kappa (u_{\theta})_x(x,t) \\
	& \qquad - 2 (u_{\theta})_{x}(x,t)(u_{\theta})_{xx}(x,t) - (u_{\theta})(x,t) (u_{\theta})_{xxx}(x,t).
	\end{aligned}
	\end{equation}
	Note that the residual is well defined and $\res_{int,\theta} \in C([0,T]\times [0,1])$ for every $\theta \in \Theta$. 
	\item Spatial boundary Residual given by,
	\begin{equation}
	\begin{aligned}
	\label{eq:hres21}
	\res_{sb1,\theta}(0,t) & := u_{\theta}(0,t), \quad \forall  t \in (0,T), \\
	\res_{sb2,\theta}(1,t) & := u_{\theta}(1,t), \quad \forall  t \in (0,T),\\
	\res_{sb3,\theta}(0,t) & := (u_{\theta})_{xx}(0,t), \quad \forall  t \in (0,T),\\
	\res_{sb4,\theta}(1,t) & := (u_{\theta})_{xx}(1,t), \quad \forall  t \in (0,T).
	\end{aligned}
	\end{equation}
	Given the fact that the neural network is smooth, this residual is well defined. 
	\item Temporal boundary Residual given by,
	\begin{equation}
	\label{eq:hres31}
	\res_{tb,\theta}(x):= \left[\bigg(u_\theta(x, 0) - u_0(x)\bigg)^2 + \bigg((u_\theta)_x(x, 0) - (u_0)_x(x)\bigg)^2\right]^{1/2}, \quad \forall x \in (0,1). 
	\end{equation}
	Again this quantity is well-defined and $\res_{tb,\theta} \in C^2((0,1))$ as both the initial data and the neural network are smooth. Notice that this temporal boundary residual differs from those of KdV and Kawahara equation by an additional spatial derivative term. This term stems from the mixed derivative $u_{txx}$ in \eqref{eq:vscl} which will come clear in the derivation of \eqref{eq:CH_hat_eq2} and \eqref{eq:CH_hat_eq3}.
\end{itemize}
These lead to the following loss function for training the PINN for approximating the Camassa-Holm equation \eqref{eq:vscl},
\begin{equation}
    \label{eq:blf}
    J(\theta):= \sum\limits_{n=1}^{N_{tb}} w^{tb}_n|\res_{tb,\theta}(x_n)|^2 + \sum\limits_{n=1}^{N_{sb}} \sum\limits_{i=1}^{4} w^{sb}_n|\res_{sbi,\theta}(t_n)|^2  + \lambda \sum\limits_{n=1}^{N_{int}} w^{int}_n|\res_{int,\theta}(x_n,t_n)|^2 .
\end{equation}
Here $w^{tb}_n$ are the $N_{tb}$ quadrature weights corresponding to the temporal boundary training points $\train_{tb}$, $w^{sb}_n$ are the $N_{sb}$ quadrature weights corresponding to the spatial boundary training points $\train_{sb}$ and $w^{int}_n$ are the $N_{int}$ quadrature weights corresponding to the interior training points $\train_{int}$. Furthermore, $\lambda$ is a hyperparameter for balancing the residuals, on account of the PDE and the initial and boundary data, respectively. 
\subsection{Bounds on the Generalization Error.}
As in the case of the KdV-Kawahara equation, we will leverage the stability of classical solutions of the Camassa-Holm equation \eqref{eq:vscl} in order to bound the PINN generalization error,
\begin{equation}
    \label{eq:begen}
    \er_{G}:= \left(\int\limits_0^T \int\limits_0^1 |u(x,t) - u^{\ast}(x,t)|^2 dx dt \right)^{\frac{1}{2}},
\end{equation}
in terms of the \emph{training error},
\begin{equation}
    \label{eq:betrain}
    \begin{aligned}
    \er^2_{T}&:=
    \lambda\underbrace{\sum\limits_{n=1}^{N_{int}} w^{int}_n|\res_{int,\theta^{\ast}}(x_n,t_n)|^2}_{(\er_T^{int})^2} +\underbrace{\sum\limits_{n=1}^{N_{tb}} + w^{tb}_n|\res_{tb,\theta^{\ast}}(x_n)|^2}_{(\er_T^{tb})^2} + \underbrace{\sum\limits_{n=1}^{N_{sb}} \sum\limits_{i=1}^{4} w^{sb}_n|\res_{sbi,\theta^{\ast}}(t_n)|^2}_{(\er_T^{sb})^2},
\end{aligned}
\end{equation}
readily computed from the training loss \eqref{eq:blf} \emph{a posteriori}. We have the following estimate,
\begin{theorem}
\label{thm:burg}
Let $\kappa > 0$ and let $u \in C^3((0,T) \times (0,1))$ be the unique classical solution of Casamma-Holm equation \eqref{eq:vscl}. Let $u^{\ast} = u_{\theta^{\ast}}$ be the PINN, generated by algorithm \ref{alg:PINN}, with loss function \eqref{eq:blf}. Then, the generalization error \eqref{eq:begen} is bounded by,
	
\begin{equation}
\label{001}
		\begin{aligned}
			\epsilon_G &\leq C_1\big(\epsilon_T^{tb} + \epsilon_T^{int} + C_2(\epsilon_T^{sb}) + C_3(\epsilon_T^{sb})^{1/2} \\
			&+ (C_{quad}^{tb})^{1/2} N_{tb}^{-\alpha_{tb} / 2} + (C_{quad}^{int})^{1/2} N_{int}^{-\alpha_{int} / 2} + C_2(C_{quad}^{sb})^{1/2} N_{sb}^{-\alpha_{sb} / 2} + C_3(C_{quad}^{sb})^{1/4} N_{sb}^{-\alpha_{sb} / 4}\big),
		\end{aligned}
	\end{equation}
	where 
	\begin{equation}
		\begin{aligned}
			C_1 &= \sqrt{T + 2C_4T^2e^{2C_4T}}, \\
			C_2 &= \sqrt{2(|\kappa| + \Vert u^* \Vert_{C_t^0C_x^2} + \Vert u \Vert_{C_t^0C_x^2})}, \\
			C_3 &= 2T^{1/4}\sqrt{2\Vert u^* \Vert_{C_t^1C_x^1} + 2\Vert u \Vert_{C_t^1C_x^1} + 2\Vert u \Vert_{C_t^0C_x^1}(\Vert u^* \Vert_{C_t^0C_x^1}+ \Vert u \Vert_{C_t^0C_x^1})}, \\
			C_4 &= \frac{1}{2} + 3\Vert u^* \Vert_{C_t^0C_x^1} + \frac{3}{2}\Vert u \Vert_{C_t^0C_x^3},
		\end{aligned}
	\end{equation}	
and $C^{tb}_{quad} = C^{tb}_{qaud}\left(\|\res_{tb,\theta^{\ast}}\|_{C^2}\right)$, $C^{int}_{quad} = C^{int}_{qaud}\left(\|\res_{int,\theta^{\ast}}\|_{C^0}\right)$, and $C^{sb}_{quad} = C^{sb}_{qaud}\left(\|\res_{sb,\theta^{\ast}}\|_{C^{1}}\right)$ are the constants associated with the quadrature errors are constants are appear in the bounds on quadrature error \eqref{eq:hquad1}-\eqref{eq:hquad3}. 
\end{theorem}
\begin{proof}
Let $\hat{u} = u^{\ast}-u$ be the error with the PINN. From the PDE \eqref{eq:vscl} and the definition of the interior residual \eqref{eq:hres11}, we have,
\begin{align}
\label{eq:CH_hat_eq}
\hat{u}_t - \hat{u}_{txx} + 2 \kappa \hat{u}_x + 3 (u^{\ast}u^{\ast}_x -uu_x)= 2 u^{\ast}_{x}u^{\ast}_{xx} - 2u_xu_{xx} + u^{\ast}u^{\ast}_{xxx} -u u_{xxx} + \res_{int}.
\end{align}
Observe also that
\begin{equation}
    \begin{aligned}
    \label{eq:CH_id}
    u^{\ast} u^{\ast}_x - uu_x  =  \hat{u}  \hat{u}_x  + u  \hat{u}_x +  \hat{u} u_x; & \quad  u^{\ast}_x u^{\ast}_{xx} - u_xu_{xx} =  \hat{u}_x  \hat{u}_{xx}  + u_x  \hat{u}_{xx} +  \hat{u}_x u_{xx}, \\
    u^{\ast} u^{\ast}_{xxx} - uu_{xxx}  &=  \hat{u}  \hat{u}_{xxx} + u  \hat{u}_{xxx} +  \hat{u} u_{xxx}.
    \end{aligned}
\end{equation}
Multiplying both sides of \eqref{eq:CH_hat_eq} with $\hat{u}$, integrating by part and using the identities \eqref{eq:CH_id} we arrive at,
	\begin{equation}
		\frac{1}{2}\frac{d}{dt}\int_0^1 (\hat{u}^2 + (\hat{u}_x)^2) \,dx + \left.\kappa\hat{u}^2\right\vert_0^1 - \left. \hat{u}\hat{u}_{tx} \right\vert_0^1 = -3\mathcal{A} + 2\mathcal{B} + \mathcal{C} + \int_0^1 \hat{u}R_{int}\,dx,
		\label{eq:CH_hat_eq1}
	\end{equation}
	where
	\begin{equation}
	\begin{aligned}
		\mathcal{A} &:= \int_0^1 \hat{u} (\hat{u}\hat{u}_x + u\hat{u}_x + \hat{u}u_x)\,dx 
		= \int_0^1 (\hat{u}_x + u_x) \hat{u}^2\,dx + \int_0^1 \hat{u}u\hat{u}_x\,dx \\
		&= \int_0^1 (\hat{u}_x + u_x) \hat{u}^2\,dx - \frac{1}{2}\int_0^1 u_x\hat{u}^2\,dx + \left.\frac{1}{2}u\hat{u}^2\right\vert_0^1 
		= \int_0^1 (\hat{u}_x + \frac{1}{2}u_x) \hat{u}^2\,dx 
		= \int_0^1 (u^*_x - \frac{1}{2}u_x) \hat{u}^2\,dx.
	\end{aligned}
	\label{eq:CH_A}
	\end{equation}
We estimate $\mathcal{B}$ as follow
	\begin{equation}
	\begin{aligned}
	\mathcal{B} &:= \int_0^1 \hat{u} (\hat{u}_x\hat{u}_{xx} + u_x\hat{u}_{xx} + \hat{u}_xu_{xx})\,dx 
	= -\frac{1}{2}\int_0^1 \hat{u}_{xxx}\hat{u}^2\,dx + \left.\frac{1}{2}\hat{u}^2\hat{u}_{xx}\right\vert_0^1 \\
	&- \int_0^1 u_x\hat{u}_x^2\,dx + \frac{1}{2}\int_0^1 u_{xxx}\hat{u}^2\,dx - \left.\frac{1}{2}u_{xx}\hat{u}^2\right\vert_0^1 + \left.\hat{u}u_x\hat{u}_x\right\vert_0^1 
	-\frac{1}{2}\int_0^1 u_{xxx}\hat{u}^2\,dx + \left.\frac{1}{2}\hat{u}^2u_{xx}\right\vert_0^1  \\
	&= -\frac{1}{2}\int_0^1\hat{u}_{xxx}\hat{u}^2\,dx - \int_0^1 u_x\hat{u}_x^2\,dx + \left.\frac{1}{2}\hat{u}^2\hat{u}_{xx}\right\vert_0^1 + \left.\hat{u}u_x\hat{u}_x\right\vert_0^1
	\end{aligned}
	\label{eq:CH_B}
	\end{equation}
On the other hand,  $\mathcal{C}$ is given by
	\begin{equation}
	\begin{aligned}
	\mathcal{C} &:= \int_0^1 \hat{u} (\hat{u}\hat{u}_{xxx} + u\hat{u}_{xxx} + \hat{u}u_{xxx})\,dx 
	= \int_0^1 (u_{xxx} + \hat{u}_{xxx})\hat{u}^2\,dx + \int_0^1 \hat{u}u\hat{u}_{xxx}\,dx \\
	&= \int_0^1 (u_{xxx} + \hat{u}_{xxx})\hat{u}^2\,dx +\frac{3}{2}\int_0^1 u_x\hat{u}_x^2\,dx - \frac{1}{2}\int_0^1 u_{xxx}\hat{u}^2\,dx - \left.\hat{u}u_x\hat{u}_x\right\vert_0^1 \\
	&= \int_0^1 (u^*_{xxx} - \frac{1}{2}u_{xxx})\hat{u}^2\,dx +\frac{3}{2}\int_0^1 u_x\hat{u}_x^2\,dx - \left.\hat{u}u_x\hat{u}_x\right\vert_0^1. \\
	\end{aligned}
	\label{eq:CH_C}
	\end{equation}
 The boundary term in \eqref{eq:CH_hat_eq1} can be bounded as
	\begin{equation}
		| \left.\hat{u}\hat{u}_{tx}\right\vert_0^1  \leq (\Vert u^* \Vert_{C_t^1C_x^1} + \Vert u \Vert_{C_t^1C_x^1})(|\res_{sb1}| + | \res_{sb2} |).
		\label{eq:CH_bt}
	\end{equation}
	From \eqref{eq:CH_hat_eq1}-\eqref{eq:CH_bt}, we get
	\begin{equation}
		\begin{aligned}
			\frac{1}{2}\frac{d}{dt}\int_0^1 (\hat{u}^2 + (\hat{u}_x)^2) \,dx &= -\left.\kappa\hat{u}^2\right\vert_0^1 + \left.\hat{u}\hat{u}_{tx} \right\vert_0^1 - 3\mathcal{A} + 2\mathcal{B} + \mathcal{C} + \int_0^1 \hat{u}\res_{int}\,dx \\
			&= \int_0^1 (-3u_x^* + \frac{3}{2}u_x + \frac{1}{2}u_{xxx})\hat{u}^2 \,dx - \frac{1}{2}\int_0^1u_x\hat{u}_x^2 \,dx + \int_0^1 \hat{u}\res_{int}\,dx \\
		    &- \left.\kappa\hat{u}^2\right\vert_0^1 + \left.\hat{u}\hat{u}_{tx} \right\vert_0^1 + \left.\hat{u}^2\hat{u}_{xx}\right\vert_0^1 + \left.\hat{u}u_x\hat{u}_x\right\vert_0^1  \\
			&\leq (\frac{1}{2} + 3\Vert u^* \Vert_{C_t^0C_x^1} + \frac{3}{2}\Vert u \Vert_{C_t^0C_x^3})\int_0^1 \hat{u}^2 \,dx + \frac{1}{2}\Vert u \Vert_{C_t^0C_x^1}\int_0^1 \hat{u}_x^2 \,dx \\
			&+ (|\kappa| + \Vert u^* \Vert_{C_t^0C_x^2} + \Vert u \Vert_{C_t^0C_x^2})(\res_{sb1}^2 + \res_{sb2}^2) + \frac{1}{2}\int_0^1 \res_{int}^2 \,dx \\
			&+ \big(\Vert u^* \Vert_{C_t^1C_x^1} + \Vert u \Vert_{C_t^1C_x^1} + \Vert u \Vert_{C_t^0C_x^1}(\Vert u^* \Vert_{C_t^0C_x^1}+ \Vert u \Vert_{C_t^0C_x^1})\big)(|\res_{sb1}| + |\res_{sb2}|) \\
			&=: C_1\sum\limits_{i=1}^4|\res_{sbi}| + C_2\sum\limits_{i=1}^4\res_{sb, i}^2 + \frac{1}{2}\int_0^1 \res_{int}^2 \,dx + C_3\int_0^1 (\hat{u}^2 + \hat{u}_x^2) \,dx.
		\end{aligned}
		\label{eq:CH_hat_eq2}
	\end{equation}
Then integrating the above inequality over $[0,\bar{T}]$ for any $\bar{T} \leq T$, we obtain
	\begin{equation}
		\begin{aligned}
			&\int_0^1 (\hat{u}^2 + \hat{u}_x^2)(x, \bar{T}) \,dx 
			\leq \int_0^1 \res_{tb}^2\,dx \\
			&+ 2C_1T^{1/2}\sum\limits_{i=1}^4(\int_0^T\res_{sbi}^2\,dt)^{1/2} + 2C_2\sum\limits_{i=1}^4(\int_0^T\res_{sbi}^2\,dt)  
			+ \int_0^T\int_0^1 \res_{int}^2 \,dxdt + 2C_3\int_0^T\int_0^1 (\hat{u}^2 + \hat{u}_x^2)\,dxdt \\
			&\leq (1 + 2C_3Te^{2C_3T})  \big(\int_0^1 \res_{tb}^2\,dx + 8C_1T^{1/2}(\sum\limits_{i=1}^4\int_0^T\res_{sbi}^2\,dt)^{1/2} 
			+ 2C_2\sum\limits_{i=1}^4(\int_0^T\res_{sbi}^2\,dt) + \int_0^T\int_0^1 \res_{int}^2 \,dxdt \big). 
		\end{aligned}
		\label{eq:CH_hat_eq3}
	\end{equation}
We can now exploit Cauchy-Schwarz and Gronwall's inequalities and integrate \eqref{eq:CH_hat_eq3} over $[0, T]$ with respect to $\bar{T}$  in order to obtain
	\begin{equation}
		\begin{aligned}
			&\epsilon_G^2 := \int_0^T\int_0^1 \hat{u}(x, \bar{T})^2\,dxd\bar{T} 
			\leq \int_0^T\int_0^1 (\hat{u}^2 + \hat{u}_x^2)(x, \bar{T}) \,dxd\bar{T}  \\
			&\leq (T + 2C_3T^2e^{2C_3T})  \big(\int_0^1 \res_{tb}^2\,dx + 8C_1T^{1/2}(\sum\limits_{i=1}^4\int_0^T\res_{sbi}^2\,dt)^{1/2} 
			+ 2C_2\sum\limits_{i=1}^4(\int_0^T\res_{sbi}^2\,dt) + \int_0^T\int_0^1 \res_{int}^2 \,dxdt \big),
		\end{aligned}
	\end{equation}
	with
	\begin{equation}
		\begin{aligned}
			C_1 &=  \Vert u^* \Vert_{C_t^1C_x^1} + \Vert u \Vert_{C_t^1C_x^1} + \Vert u \Vert_{C_t^0C_x^1}(\Vert u^* \Vert_{C_t^0C_x^1}+ \Vert u \Vert_{C_t^0C_x^1}), \\
			C_2 &= |\kappa| + \Vert u^* \Vert_{C_t^0C_x^2} + \Vert u \Vert_{C_t^0C_x^2}, \quad
			C_3 = \frac{1}{2} + 3\Vert u^* \Vert_{C_t^0C_x^1} + \frac{3}{2}\Vert u \Vert_{C_t^0C_x^3}. 
		\end{aligned}
	\end{equation}
The statement of the theorem can be eventually proven by finally using the estimates \eqref{eq:hquad1}, \eqref{eq:hquad2}, \eqref{eq:hquad3}.
\end{proof}
\subsection{Numerical Experiments}
We set $\kappa = k^2$ in the Camassa-Holm equation \eqref{eq:vscl} and follow \cite{peakon_lim1, peakon_lim2, peakon_lim3} to consider the following exact solution for the \emph{single soliton},
\begin{equation}
\begin{aligned}
&u(\theta) = \frac{2kcp^2}{(1+k^2p^2) + (1-k^2p^2)\cosh\theta}, \\
&\Theta = p(x-\tilde{c}t + x_0), \\
&\Theta = \frac{\theta}{k} + p\ln\frac{(1+kp) + (1-kp)e^{\theta}}{(1-kp) + (1+kp)e^{\theta}},
\end{aligned}
\end{equation}
where $\tilde{c} = \frac{c}{k} = \frac{2k^2}{1-k^2p^2}$ and $p$ is an additional parameter. To obtain the exact solution, we need to compute the inverse of $\Theta(\theta)$. $\Theta(\theta)$ is invertible if and only if $0 < kp < 1$ which is an additional constraint when choosing $p$. Note from the above formula that the soliton moves to the right with speed $\tilde{c}$, while preserving its shape during the evolution. 
\begin{table}[htbp] 
    \centering
    \renewcommand{\arraystretch}{1.1} 
    
    \footnotesize{
        \begin{tabular}{ c c c c c c c c c c} 
            \toprule
            \bfseries   &\bfseries $N_{int}$  & \bfseries $N_{sb}$& \bfseries $N_{tb}$  &\bfseries $K-1$ & \bfseries $d$  &\bfseries $\lambda$ &\bfseries $\er_T$&\bfseries $\er_G^r$ \\ 
            \midrule
            \midrule
            Single Soliton &16384   & 4096 & 4096 &4&20& 1 &3.70e-06& 0.00191\% \\
            \midrule 
            Double Soliton           &16384   & 4096 & 4096 &8&24& 0.1 &0.00127& 0.186\% \\

            \bottomrule
        \end{tabular}
    \caption{Best performing \textit{Neural Network} configurations for the single soliton and double soliton problem. Low-discrepancy Sobol points are used for every reported numerical example.}
        \label{tab:CH}
    }
\end{table}
We apply the PINNs algorithm \ref{alg:PINN} to approximate the single soliton, with parameters $k = 0.6, p = 1$. The hyperparameters, corresponding to the smallest training error during ensemble training are reported in Table \ref{tab:CH}. In figure \ref{fig:CH} (left), we plot the exact soliton and its PINN approximation, at the initial time and at a later time and observe from the figure that the trained PINN can approximate the soliton to very high accuracy. This is further validated by the extremely low generalization error, reported in Table \ref{tab:CH}. Moreover, from Table \ref{tab:CH_single_cp}, we observe that the PINN for approximating this single solution is training very fast and an error of less than $1\%$ already results from less than $500$ iterations, that corresponding to approximately $2$ minutes of training time. 
\begin{figure}[h!]
    \begin{subfigure}{.49\textwidth}
        \centering
        \includegraphics[width=1\linewidth]{{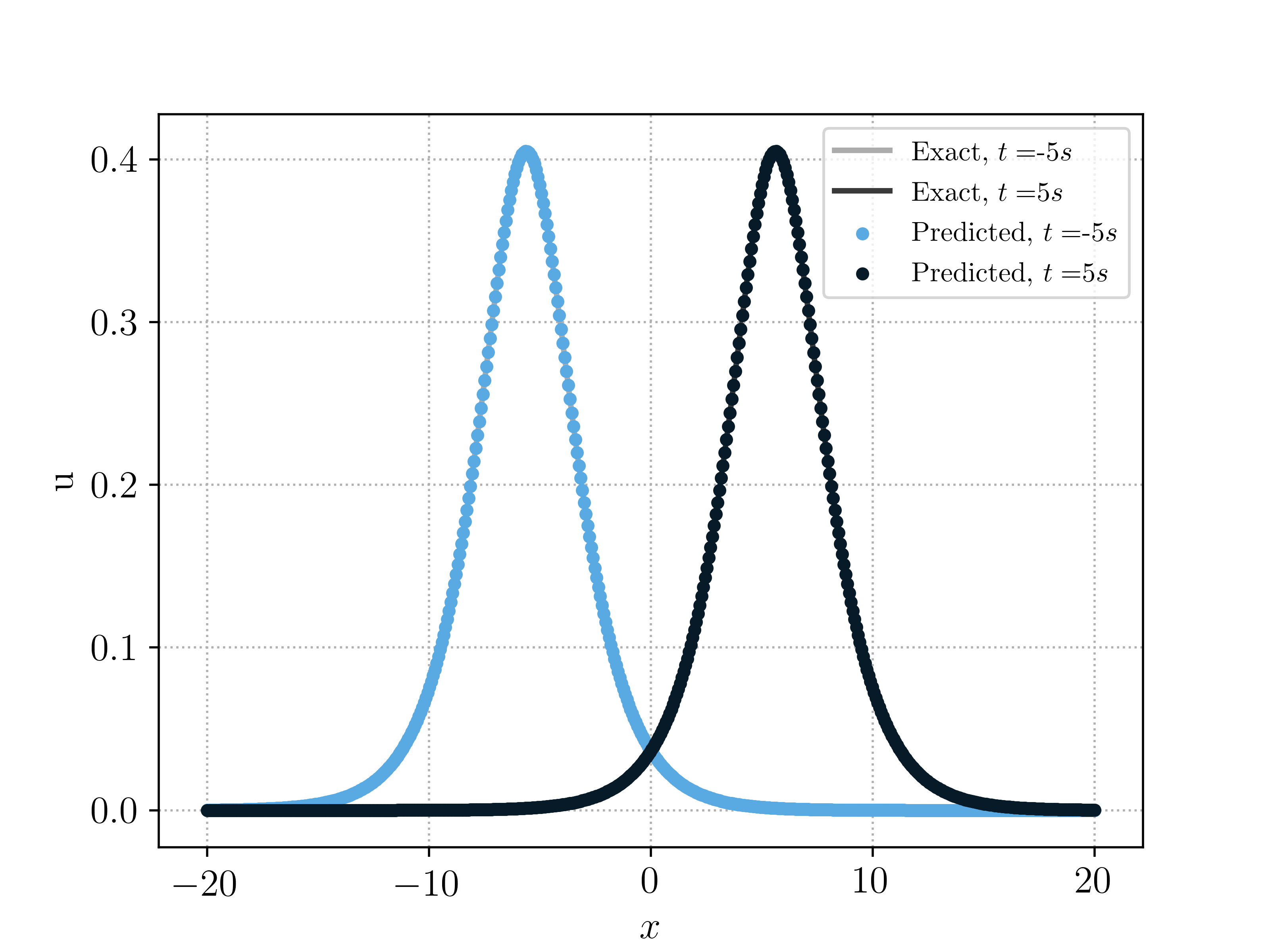}}
        \caption{Single soliton, $k=0.6, p=1$}
    \end{subfigure}
    \begin{subfigure}{.49\textwidth}
        \centering\
        \includegraphics[width=1\linewidth]{{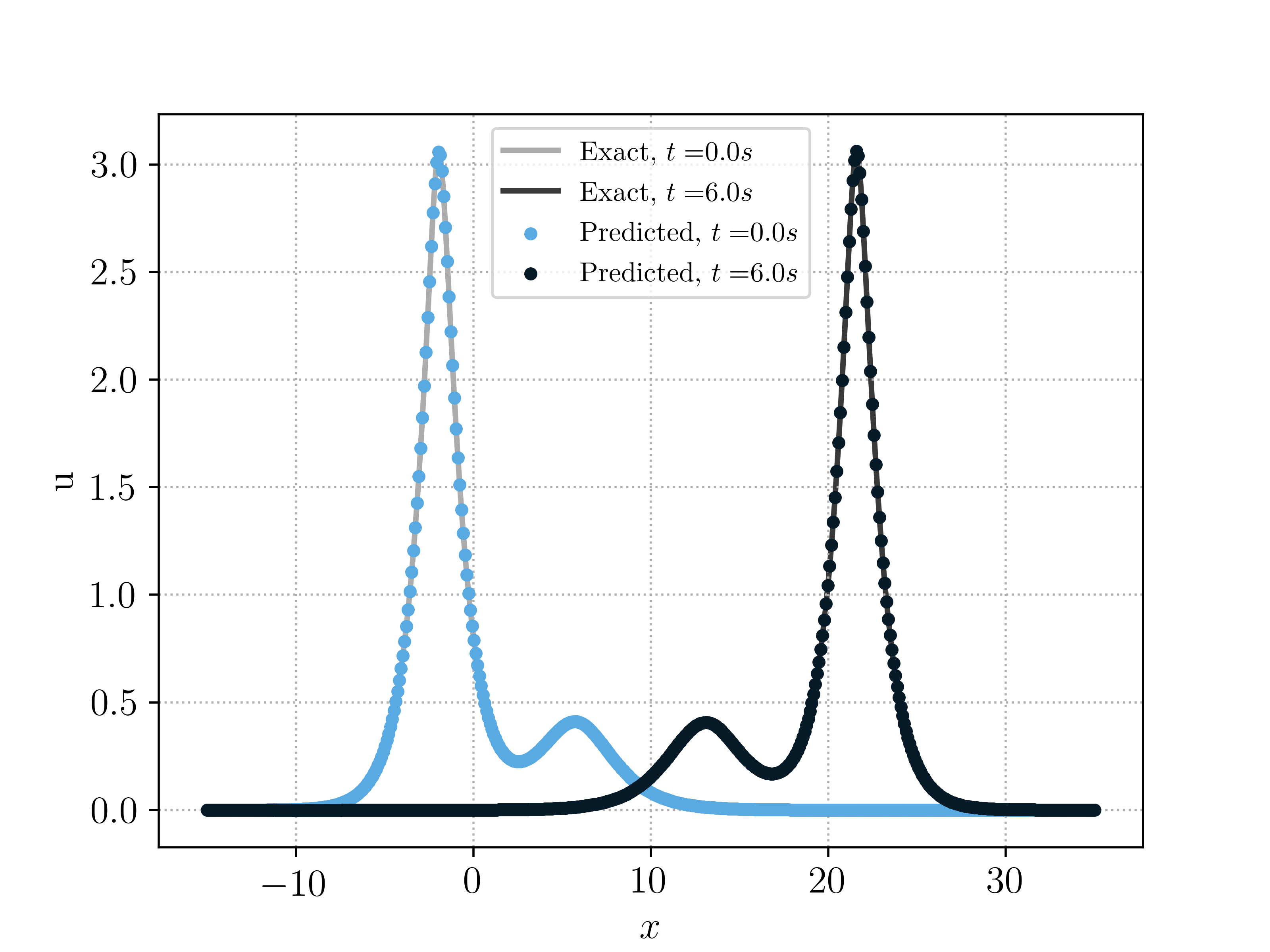}}
        \caption{Double soliton, $k = 0.6, p_1 = 1.5, p_2 = 1$}
    \end{subfigure}
    \caption{The exact and PINN solutions of single and double soliton test case of generalized CH equation.}
    \label{fig:CH}
\end{figure}

\begin{table}[htbp] 
    \centering
    \renewcommand{\arraystretch}{1.1} 
    
    \footnotesize{
        \begin{tabular}{ c c c c } 
            \toprule
            \bfseries max\_iters  &\bfseries training time$/s$ &\bfseries $\eps_T$ &\bfseries $\eps_G^r$ \\ 
            \midrule
            \midrule
         100   &    36  &   8.08e-03   &  2.81e-01 \\
         \midrule 
         500   &   161   &  4.71e-04  &   5.93e-03 \\
         \midrule 
        1000   &   284   &  1.61e-04  &   1.14e-03 \\
        \midrule 
        2000   &   560   &  4.96e-05   &  3.75e-04 \\
        \midrule 
        5000   &  1457   &  9.77e-06   &  9.31e-05 \\
        \midrule 
       10000   &  1667   &  2.83e-06   &  1.94e-05 \\
              
            \bottomrule
        \end{tabular}
    \caption{Results of different training iterations for single soliton case of CH equation.}
		\label{tab:CH_single_cp}
    }
\end{table}
Next, we again follow \cite{peakon_lim2} to consider additional parameters $p_{1,2}$ and define
\begin{equation}
	c_i = \frac{2k^3}{1-k^2p_i^2}, \quad 	w_i = -p_ic_i, \quad i = 1, 2
\end{equation}
and
\begin{equation}
	A_{12} = \frac{(p_1-p_2)^2}{(p_1+p_2)^2}.
\end{equation}
For $i = 1, 2$, we further define
\begin{equation}
\begin{aligned}
	a_i = 1 + kp_i, \quad b_i = 1 - kp_i,
\end{aligned}
\end{equation}
and as before, we define $\theta_i$ w.r.t. $y$ as
\begin{equation}
	\theta_i = p_i(y-c_it+\alpha_i), \quad i = 1, 2
\end{equation}
and 
\begin{equation}
\begin{aligned}
v_{12} &= \frac{4k^3(p_1-p_2)^2}{(1 - k^2p_1^2)(1 - k^2k_2^2)},  \quad b_{12} &= \frac{8k^6(p_1-p_2)^2(1-k^4p_1^2p_2^2)}{(1-k^2p_1^2)^2(1-k^2p_2^2)^2}
\end{aligned}
\end{equation}
Then, the exact double soliton solution w.r.t. $y$ is given by
\begin{equation}
	u(y, t) = k^2 + \frac{2}{k}\frac{w_1^2e^{\theta_1} + w_2^2e^{\theta_2} + b_{12}e^{\theta_1+\theta_2} + A_{12}(w_1^2e^{\theta_1+2\theta_2} + w_2^2e^{2\theta_1+\theta_2})}{rf^2},
\end{equation}
where
\begin{equation}
\begin{aligned}
f(y, t) &= 1 + e^{\theta_1} + e^{\theta_2} + A_{12}e^{\theta_1 + \theta_2} \\
r(y, t) &= k + \frac{2}{f^2}(c_1p_1^2e^{\theta_1} + c_2p_2^2e^{\theta_2} + v_{12}e^{\theta_1+\theta_2} + A_{12}(c_1p_1^2e^{\theta_1 + 2\theta2} + c_2p_2^2e^{2\theta_1 + \theta2})).
\end{aligned}
\end{equation}
Finally we have the following relation between $x$ and $y$
\begin{equation}
	x(y, t) = \frac{y}{k} + \ln\frac{a_1a_2 + b_1a_2e^{\theta1} + b_2a_1e^{\theta2} + b_1b_2A_{12}e^{\theta_1 + \theta_2}}{b_1b_2 + a_1b_2e^{\theta1} + q_2b_1e^{\theta2} + a_1a_2A_{12}e^{\theta_1 + \theta_2}} + k^2t + \alpha,
\end{equation}
where $\alpha$ is the phase parameter. To obtain the exact solution, we need to compute the inverse of $x(y, t)$ w.r.t. $y$ at the training points. $x(y, t)$ is invertible w.r.t. $y$ if and only if $0 < kp_i < 1, i = 1, 2$ which, again, is an additional constraint when choosing $p_1, p_2$. 

\begin{table}[htbp] 
    \centering
    \renewcommand{\arraystretch}{1.1} 
    
    \footnotesize{
        \begin{tabular}{ c c c c } 
            \toprule
            \bfseries max\_iters  &\bfseries training time$/s$ &\bfseries $\eps_T$ &\bfseries $\eps_G^r$ \\ 
            \midrule
            \midrule
         100   &    83  &   3.63e-02   &  7.19e-01 \\
         \midrule
         500   &   386  &   8.37e-03   &  1.68e-01 \\
         \midrule
        1000   &   762   &  5.52e-03  &   6.99e-02 \\
        \midrule
        2000   &  1508  &   3.10e-03  &   3.17e-02 \\
        \midrule
        5000  &   4083   &  8.71e-04  &   5.29e-03 \\
        \midrule
       10000  &   5747   &  4.09e-04   &  1.84e-03 \\
              
            \bottomrule
        \end{tabular}
    \caption{Results of different training iterations for double soliton case of CH equation.}
		\label{tab:CH_double_cp}
    }
\end{table}

We set $k=0,6,p_1=1.5,p_2=2$ in the above formula and apply the PINNs algorithm to compute the double soliton for the Camassa-Holm equation. The hyperparameters, corresponding to the smallest training error during ensemble training are reported in Table \ref{tab:CH}. In figure \ref{fig:CH} (right), we plot the exact soliton and its PINN approximation, at the initial time and at a later time and observe from the figure that the trained PINN can approximate the soliton to high accuracy. This is further validated by the very low generalization error, reported in Table \ref{tab:CH}. In particular, the ability of the PINN to resolve not just the sharp waves but also the dynamic wave interaction is noteworthy. The error as a function of the training iterations (and hence the computational cost) is shown in Table \ref{tab:CH_double_cp} and we observe that significantly more training time is necessary to resolve the double soliton than the single soliton. For instance, one needs approximately $25$ minutes of training time for obtaining an error of $3\%$. This difference in convergence of training iterations between the single soliton and double soliton cases is nicely explained from the observations in Figure \ref{fig:CH_double_cp}, where we plot the PINN solutions at a sequence of training iterations. We observe that for single soliton, the sharp peak is very quickly approximated during training. Similar, the sharp peak corresponding to the faster soliton is very quickly approximated in the double soliton case. On the other hand, the complicated wave pattern, with a crest and a through in the wake of the fast solution, takes several more training iterations to resolve. 
\begin{figure}
	\begin{subfigure}{.49\textwidth}
		\centering
		\includegraphics[width=1\linewidth]{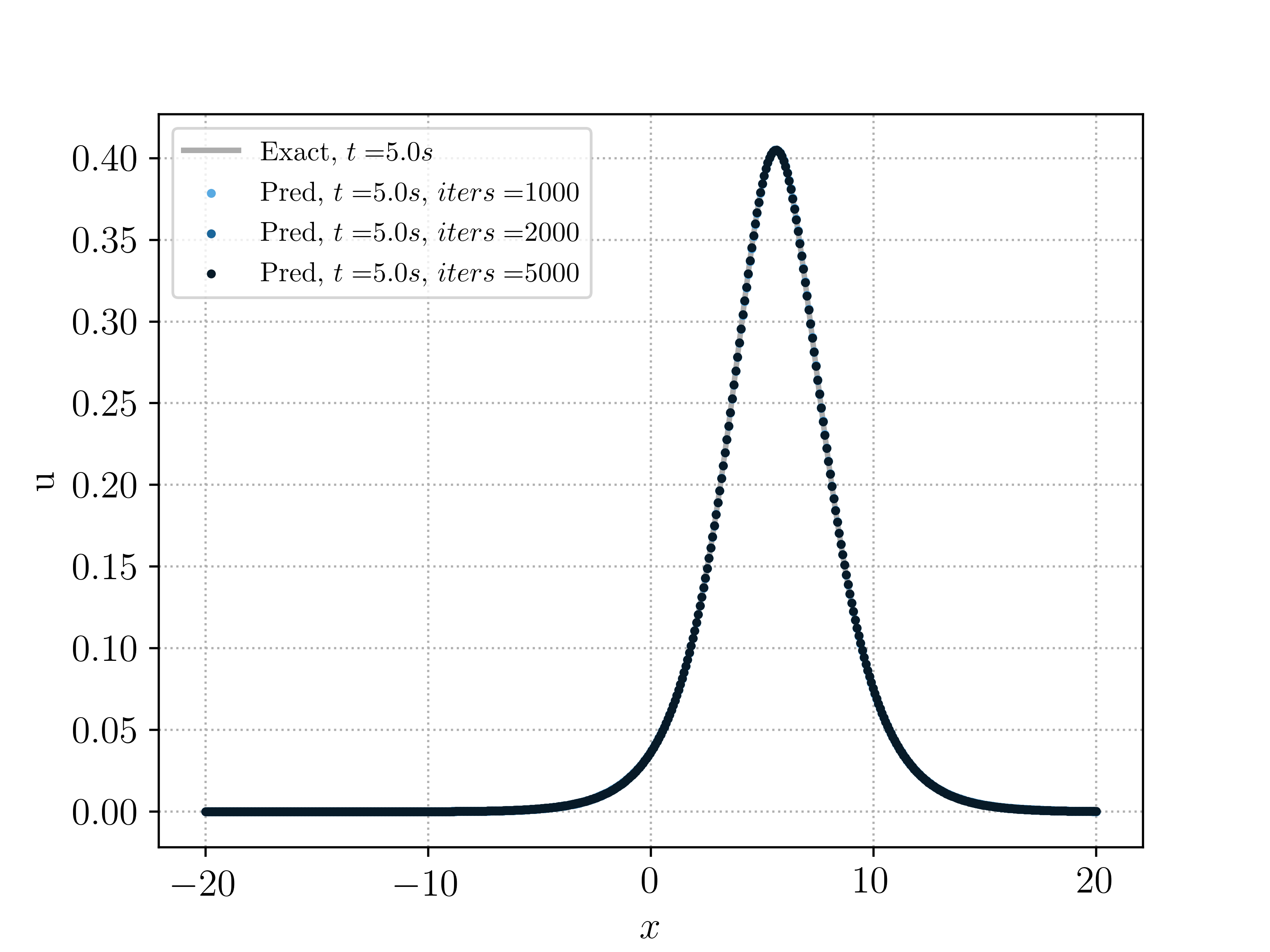}
		\caption{Single soliton}
	\end{subfigure}
	\begin{subfigure}{.49\textwidth}
		\centering
		\includegraphics[width=1\linewidth]{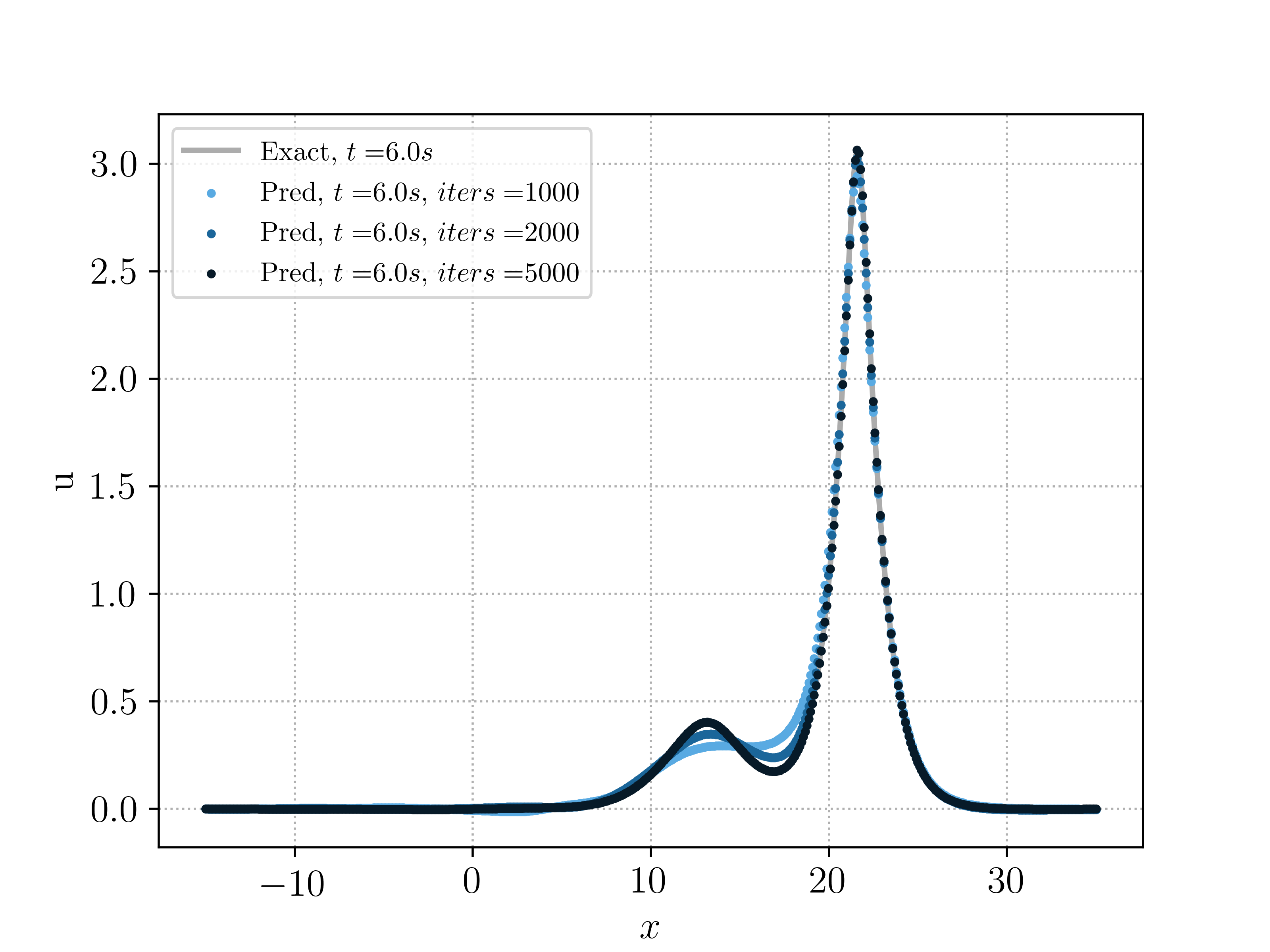}
		\caption{Double soliton}
	\end{subfigure}
	\caption{Plots of different train iterations at final time.}
	\label{fig:CH_double_cp}
\end{figure}
\begin{remark}
Taking the limit $\kappa \rightarrow 0$ in the formulas for the single soliton and the double soliton for the Camassa-Holm equation, results in the well-known single peakon and double peakon solutions of the Camassa-Holm equations \cite{Hol2}. However, peakons have a singularity in their derivatives and are at most in $H^1$. Thus, the stability result as well as the bound on the generalization error no longer hold, as they require $C^3$ regularity for the solutions. Consequently, we cannot expect to compute peakons with the current version of the PINNs algorithm. 
\end{remark}
\section{Benjamin-Ono Equation}
\label{sec:5}
\subsection{The underlying PDE}
As a final example of nonlinear dispersive PDEs, we consider the following Benjamin-Ono (BO) equation
\begin{equation}
\label{eq:ie1}
\begin{aligned}
u_t + uu_x + Hu_{xx} & =0, \quad x \in \R, \quad t>0, \\
u(x,0)&=u_0(x), \quad x \in \R, \\
u(x,t) &=u(x+1,t), \quad x \in \R, \quad t>0,
\end{aligned}
\end{equation}
with $H$ denoting the \emph{Hilbert transform} 
defined by the principle value integral
\begin{equation*}
  H u(x) := \mathrm{p.v.} \, \frac{1}{\pi} \int_{\R} \frac{u(x-y)}{y} \,dy.
\end{equation*}
The BO equation was first deduced by Benjamin \cite{benjamin} and Ono \cite{ono} as an approximate model for long-crested unidirectional waves at the interface of a two-layer system of incompressible inviscid fluids, one being infinitely deep. Later, it was shown to be a completely integrable system. In the periodic setting, Molinet \cite{molinet3a} proved well-posedness in $H^s (\mathbb{T})$ for  $s \ge0$. We recall the following well-posedness result for the classical solutions of the BO equation,
\begin{theorem}
	\label{0011}
	For any $s>5/3$, let $u_0 \in H^{s}(0,1)$. Then there exists a global smooth solution to \eqref{eq:ie1} such that
	$$
	u \in C(0,T; H^{s}(0,1)), \quad u_t \in C^1(0,T; H^{s-2}(0,1)).
	$$
\end{theorem}
Note that the above result was also used by Kenig, Ponce and Vega \cite{kenig} to prove uniqueness properties of BO equation. Moreover, the above result ensures that the solutions satisfy the equation \eqref{eq:ie1} pointwise for sufficiently smooth initial data.
\subsection{PINNs}
To specify the PINNs algorithm for the BO equation \eqref{eq:ie1}, we start by choosing the training set as in section \ref{sec:train}. We define the residual $\res$ in algorithm \ref{alg:PINN}, consisting of the following parts,
\begin{itemize}
	\item \emph{Interior residual} given by,
	\begin{equation}
	\label{eq:ires11}
	\res_{int,\theta}(x,t):= (u_{\theta})_t(x,t) + u_{\theta}(x,t) (u_{\theta})_{x}(x,t) + H (u_{\theta})_{xx}(x,t), \quad (x,t) \in (0,1) \times (0,T),
	\end{equation}
	\item  \emph{Spatial boundary Residual} given by,
	\begin{equation}
	\label{eq:ires31}
	\res_{sb,\theta}(x,t):= u_{\theta}(x,t)- u_{\theta}(x+1,t), \quad \forall x \in \R, ~ t \in (0,T]. 
	\end{equation}
	\item \emph{Temporal boundary Residual} given by,
	\begin{equation}
	\label{eq:ires41}
	\res_{tb,\theta}(x):= u_{\theta}(x,0) - u_0(x), \quad \forall x \in (0,1). 
	\end{equation}
\end{itemize}
Next, we consider the following loss function for training PINNs to approximate the BO equation \eqref{eq:ie1},
\begin{equation}
\label{eq:ilf1}
J(\theta):= \sum\limits_{n=1}^{N_{tb}} w^{tb}_n|\res_{tb,\theta}(x_n)|^2 + \sum\limits_{n=1}^{N_{sb}} w^{sb}_n|\res_{sb,\theta}(x_n,t_n)|^2 + \lambda  \sum\limits_{n=1}^{N_{int}} w^{int}_n|\res_{int,\theta}(x_n,t_n)|^2.
\end{equation}
Here the residuals are defined by \eqref{eq:ires11}-\eqref{eq:ires41}. $w^{tb}_n$ are the $N_{tb}$ quadrature weights corresponding to the temporal boundary training points $\train_{tb}$, $w^{sb}_n$ are the $N_{sb}$ quadrature weights corresponding to the spatial boundary training points $\train_{sb}$ and $w^{int}_n$ are the $N_{int}$ quadrature weights corresponding to the interior training points $\train_{int}$. Furthermore, $\lambda$ is a hyperparameter for balancing the residuals, on account of the PDE and the initial and boundary data, respectively. 
\subsection{Estimate on the generalization error.}
We denote the PINN, obtained by the algorithm \ref{alg:PINN}, for approximating the BO equation, as $\bu^{\ast}= \bu_{\theta^{\ast}}$, with $\theta^{\ast}$ being a (approximate, local) minimum of the loss function \eqref{eq:lf2},\eqref{eq:ilf1}. We consider the following generalization error,
\begin{equation}
\label{eq:iegen1}
\er_{G}:= \left(\int\limits_0^T \int\limits_0^1 \|\bu(x,t) - \bu^{\ast}(x,t)\|^2 dx dt \right)^{\frac{1}{2}},
\end{equation}
with $\|\cdot\|$ denoting the Euclidean norm in $\R^d$. We will bound the generalization error in terms of the following \emph{training errors},
\begin{equation}
\label{eq:ietrain1}
\er_T^2:= \underbrace{\sum\limits_{n=1}^{N_{tb}} w^{tb}_n|\res_{tb,\theta^{\ast}}(x_n)|^2}_{\left(\er_T^{tb}\right)^2} + \underbrace{\sum\limits_{n=1}^{N_{sb}} w^{sb}_n|\res_{sb,\theta^{\ast}}(x_n,t_n)|^2}_{\left(\er_T^{sb}\right)^2} +\lambda\underbrace{\sum\limits_{n=1}^{N_{int}} w^{int}_n|\res_{int,\theta^{\ast}}(x_n,t_n)|^2}_{\left(\er_T^{int}\right)^2}.
\end{equation}
As in the previous sections, the training errors can be readily computed \emph{a posteriori} from the loss function \eqref{eq:ilf1}. 

We have the following bound on the generalization error in terms of the training error,
\begin{theorem}
	\label{thm:euler1}
	Let $u \in C^3([0,1] \times [0,T])$ be the unique classical solution of Benjamin-Ono equation \eqref{eq:ie1}. Let $u^{\ast} = u_{\theta^{\ast}}$ be the PINN, generated by algorithm \ref{alg:PINN}, with loss function \eqref{eq:ilf1}. Then, the generalization error \eqref{eq:iegen1} is bounded by,
	\begin{equation}
	\label{0001}
	\begin{aligned}
	\epsilon_G &\leq C_1\big(\epsilon_T^{tb} + \epsilon_T^{int} + C_2(\epsilon_T^{sb})^{1/2} \\
	&+ (C_{quad}^{tb})^{1/2} N_{tb}^{-\alpha_{tb} / 2} + (C_{quad}^{int})^{1/2} N_{int}^{-\alpha_{int} / 2} + C_2(C_{quad}^{sb})^{1/4} N_{sb}^{-\alpha_{sb} / 4}\big),
	\end{aligned}
	\end{equation}
	where 
	\begin{equation}
	\begin{aligned}
	C_1 &= \sqrt{T + 2C_3T^2e^{2C_3T}}, \\
	C_2 &= T^{1/4}\sqrt{2(\Vert u^* \Vert_{C_t^0C_x^2} + \Vert u \Vert_{C_t^0C_x^2}) + 2\Vert u \Vert_{C_t^0C_x^0}(\Vert u \Vert_{C_t^0C_x^0} + \Vert u^* \Vert_{C_t^0C_x^0})}, \\
	C_3 &= \frac{1}{2} + \Vert u^* \Vert_{C_t^0C_x^1} + \frac{1}{2}\Vert u \Vert_{C_t^0C_x^1}, \\
	\end{aligned}
	\end{equation}
and $C^{tb}_{quad} = C^{tb}_{qaud}\left(\|\res_{tb,\theta^{\ast}}\|_{C^3}\right)$, $C^{int}_{quad} = C^{int}_{qaud}\left(\|\res_{int,\theta^{\ast}}\|_{C^{1}}\right)$, and $C^{sb}_{quad} = C^{sb}_{qaud}\left(\|\res_{sb,\theta^{\ast}}\|_{C^3}\right)$ are the constants associated with the quadrature errors \eqref{eq:hquad1}-\eqref{eq:hquad3}. 
\end{theorem}
\begin{proof}
	We will drop explicit dependence of all quantities on the parameters $\theta^{\ast}$ for notational convenience. We denote the difference between the underlying solution $u$ of \eqref{eq:ie1} and PINN $u^{\ast}$ as $\hat{u} = u^{\ast} - u$. Using the PDE \eqref{eq:ie1} and the definitions of the residuals \eqref{eq:ires11}-\eqref{eq:ires41}, a straightforward calculation yields the following PDE for the $\hat{u}$,
	\begin{equation}
	\label{eq:iehat1}
	\begin{aligned}
	\hat{u}_t + H\hat{u}_{xx} +u^{\ast} u^{\ast}_x- u u_x&= \res_{u}, \quad \mbox{a.e.}\,(x,t) \in (0,1) \times (0,T), \\
	\hat{u}(0,t) - \hat{u}(1,t) &= \res_{sb}, \quad t \in (0,T), \\
	\hat{u}(x,0) &= \res_{tb}, \quad x \in (0,1).
	\end{aligned}
	\end{equation}
	We take a inner product of the equation in \eqref{eq:iehat1} with the vector $\hat{u}$, and integrate by parts to obtain the term coming from the Hilbert transform
	\begin{align*}
	\int_0^1 \hat{u} H(\hat{u}_{xx}) \,dx = - \int_0^1 \hat{u}_x H(\hat{u}_x)\,dx + H(\hat{u}_x)(1) \hat{u}(1) - H(\hat{u}_x)(0) \hat{u}(0) = H(\hat{u}_x)(1) \hat{u}(1) - H(\hat{u}_x)(0) \hat{u}(0).
	\end{align*}
	$\int_0^1 \hat{u}_x H(\hat{u}_x)\,dx$ vanishes because Hilbert transform is anti-symmetric w.r.t L2 inner product. The boundary terms on the other hand can be bounded as follows,
	\begin{align*}
	\Big[H(\hat{u}_x)(1) \hat{u}(1) - H(\hat{u}_x)(0) \hat{u}(0) \Big] 
	&= \Big[ \big(H(\hat{u}_x)(1) - H(\hat{u}_x)(0) \big) \hat{u}(1) + H(\hat{u}_x)(0) (\hat{u}(1) -\hat{u}(0))\Big] \\
	&= H(\hat{u}_x)(0) (\hat{u}(1) -\hat{u}(0)) 
	\le \| H(\hat{u}_x)(0) \|_{C^0_t} |\res_{sb}| 
    \le (\| u \|_{C^0_tC^2_x} + \| u^* \|_{C^0_tC^2_x}) |\res_{sb}|.
	\end{align*}
In the second line, we have exploited the periodicity of $u$ and $u^*$. For the remaining terms, we can follow the arguments given before and get
\begin{equation}
	\begin{aligned}
	\frac{1}{2}\frac{d}{dt}\int_0^1 \hat{u}^2\,dx 
	&= - \int_0^1 \hat{u}H\hat{u}_{xx}\,dx - \int_0^1\hat{u}(\hat{u}\hat{u}_x - u\hat{u}_x + u_x\hat{u})\,dx + \int_0^1 \hat{u}\res_{int}\,dx \\
	&\leq (\Vert u^* \Vert_{C_x^2} + \Vert u \Vert_{C_x^2})|\res_{sb}| 
	- \int_0^1 (u^*_x - \frac{1}{2}u_x)\hat{u}^2 - \left.\frac{1}{2}u\hat{u}^2\right\vert_0^1 
	+ \int_0^1 \hat{u}\res_{int}\,dx \\
	&\leq (\Vert u^* \Vert_{C_x^2} + \Vert u \Vert_{C_x^2})|\res_{sb} | \\
	&+ (\Vert u^* \Vert_{C_x^1} + \frac{1}{2}\Vert u \Vert_{C_x^1})  \int_0^1 \hat{u}^2\,dx + \Vert u \Vert_{C_x^0}(\Vert u \Vert_{C_x^0} + \Vert u^* \Vert_{C_x^0})|\res_{sb}| \\
	&+ \frac{1}{2}\int_0^1 \res_{int}^2 \,dx + \frac{1}{2}\int_0^1 \hat{u}^2 \,dx \\
	&\leq \big(\Vert u^* \Vert_{C_t^0C_x^2} + \Vert u \Vert_{C_t^0C_x^2} + \Vert u \Vert_{C_t^0C_x^0}(\Vert u \Vert_{C_t^0C_x^0} + \Vert u^* \Vert_{C_t^0C_x^0})\big)|\res_{sb}| \\
	&+ \frac{1}{2}\int_0^1 \res_{int}^2 \,dx + (\frac{1}{2} + \Vert u^* \Vert_{C_t^0C_x^1} + \frac{1}{2}\Vert u \Vert_{C_t^0C_x^1})  \int_0^1 \hat{u}^2\,dx  \\
	&=: C_1|\res_{sb}| + \frac{1}{2}\int_0^1 \res_{int}^2 \,dx + C_2\int_0^1 \hat{u}^2\,dx.
	\end{aligned}
\end{equation}
Then integrating the above inequality over $[0,\bar{T}]$ for any $\bar{T} \leq T$ and using Cauchy-Schwarz and Gronwall's inequalities  we obtain
	\begin{equation}
	\begin{aligned}
	\int_0^1 \hat{u}(x, \bar{T})^2\,dx 
	&\leq \int_0^1 \res_{tb}^2\,dx + 2C_1T^{1/2}(\int_0^T\res_{sb}^2\,dt)^{1/2} + \int_0^T\int_0^1 \res_{int}^2 \,dxdt + 2C_2\int_0^{\bar{T}}\int_0^1 \hat{u}^2\,dxdt \\
	&\leq (1 + 2C_2Te^{2C_2T})  \big(\int_0^1 \res_{tb}^2\,dx + C_1T^{1/2}(\int_0^T\res_{sb}^2\,dt)^{1/2} + \int_0^T\int_0^1 \res_{int}^2 \,dxdt \big).
	\end{aligned}
	\label{eq:BO_hat_eq3}
	\end{equation}
     Finally, we integrate \eqref{eq:BO_hat_eq3} over $\bar{T} \in [0, T]$ and arrive at
	\begin{equation}
	\begin{aligned}
	\epsilon_G^2 &:= \int_0^T\int_0^1 \hat{u}(x, \bar{T})^2\,dxd\bar{T} \\
	&\leq (T + 2C_2T^2e^{2C_2T})  \big(\int_0^1 \res_{tb}^2\,dx + 2C_1T^{1/2}(\int_0^T\res_{sb}^2\,dt)^{1/2} + \int_0^T\int_0^1 \res_{int}^2 \,dxdt \big),
	\end{aligned}
	\end{equation}
	with
	\begin{equation}
	\begin{aligned}
	C_1 =  \Vert u^* \Vert_{C_t^0C_x^2} + \Vert u \Vert_{C_t^0C_x^2} + \Vert u \Vert_{C_t^0C_x^0}(\Vert u \Vert_{C_t^0C_x^0} + \Vert u^* \Vert_{C_t^0C_x^0}), \quad
	C_2 = \frac{1}{2} + \Vert u^* \Vert_{C_t^0C_x^1} + \frac{1}{2}\Vert u \Vert_{C_t^0C_x^1}.
	\end{aligned}
	\end{equation}
The proof of theorem can be eventually attained  by applying the estimates \eqref{eq:hquad1}, \eqref{eq:hquad2}, \eqref{eq:hquad3}.

\end{proof}
\subsection{Evaluation of the singular integral}
Note that the in the PINNs algorithm for approximating the BO equation as well as in the derivation of the above error bound, we have assumed that the Hilbert transform in \eqref{eq:ie1} can be evaluated exactly. In practice, this is not possible and we need to approximate the Hilbert transform. To this end, we focus on the periodic case. The periodic Hilbert transform is defined by
\begin{equation}
	H_{per}u(x) = \textrm{p.v.} \frac{1}{2L} \int_{-L}^L \cot(\frac{\pi}{2L}y)u(x - y)dy.
\end{equation}
To compute the above \emph{non-local} term, we use a Cartesian grid $\{ x_i \}_{i = -N}^N$ and additionally require $x_0 = 0$. And we can discretize the singular integral term as
\begin{equation}
\label{eq:ehr}
\begin{aligned}
	H_{per}u_{xx}(x) &= \textrm{p.v.} \frac{1}{2L} \int_{-L}^L \cot(\frac{\pi}{2L}y)u_{xx}(x - y)dy \\
	&\approx \frac{1}{2N} \sum\limits_{j=-N, j \neq 0}^N \cot(\frac{\pi}{2L}x_j)u_{xx}(x - x_j).
\end{aligned}
\end{equation}
We exclude index $j = 0$ in order to be consistent with the definition of principal value because $x_0 = 0$ is a singularity of $\cot(\frac{\pi}{2L}x_j)$. 

More importantly, what we need to compute is the term, $\left.H_{per}u_{xx}(x)\right\vert_{x_i}$ which can be represented as a discrete periodic convolution of $\cot(\frac{\pi}{2L}x_j)$ and $\left.u_{xx}(x)\right\vert_{x_j}$ 
\begin{equation}
\label{eq:ehr1}
\begin{aligned}
H_{per}u_{xx}(x_i) &\approx \frac{1}{2N} \sum\limits_{j=-N, j \neq 0}^N \cot(\frac{\pi}{2L}x_j)u_{xx}(x_i - x_j) \\
&= \frac{1}{2N} \sum\limits_{j=-N, j \neq 0}^N \cot(\frac{\pi}{2L}x_j)u_{xx}(x_{i - j}).
\end{aligned}
\end{equation}
This implies that to compute $H_{per}u_{xx}(x_i), -N \leq i \leq N$ we only need to compute $u_{xx}(x_i), -N \leq i \leq N$. Moreover, the discrete periodic convolution \eqref{eq:ehr1} can be accelerated by a Fast Fourier transform(FFT) to obtain a complexity of $O(N\log(N))$.
\begin{table}[htbp] 
    \centering
    \renewcommand{\arraystretch}{1.1} 
    
    \footnotesize{
        \begin{tabular}{ c c c c c c c c c c c} 
            \toprule
            \bfseries   &\bfseries $N_{int}$  & \bfseries $N_{sb}$& \bfseries $N_{tb}$  &\bfseries $K-1$ & \bfseries $d$ &\bfseries $\lambda$ &\bfseries $\er_T$&\bfseries $\er_G^r$ &\bfseries $\Delta$ \\ 
            \midrule
            \midrule
            Single Soliton &32768   & 8192 & 8192 & 12& 24 & 1 &0.000296& 0.773\% & 4 \\
            \midrule 
            Double Soliton           &65536   & 16384 & 16384 & 4 & 20 & 10 &0.00616 & 0.657\% & 30 \\

            \bottomrule
        \end{tabular}
    \caption{Best performing \textit{Neural Network} configurations for the periodic single soliton and real-line double soliton problem. Low-discrepancy Sobol points are used for all boundary points; Cartesian grids are used for all interior points.}
        \label{tab:BO}
    }
\end{table}

\subsection{Numerical experiments}
In addition to the previous hyperparameters, an additional one $\Delta = \frac{\Delta t}{\Delta x}$ i.e., the ratio of the time and space steps on the space-time Cartesian grid, also needs to be set for the BO equation and we select it through ensemble training. We start with the periodic single soliton test case with the exact solution,
\begin{equation}
u(x, t) = \frac{2c\delta^2}{1 - \sqrt{1-\delta^2}\cos(c\delta(x-ct-x_0))}, \quad \delta = \frac{\pi}{cL},
\end{equation}
where $L$ is the half periodicity. This represents a single bump moving to the right with speed $c$ periodically with initial peak at $x = x_0$. In our experiments, we choose $L = 15$, $c = 0.25$ and $x_0 = 0$. The selected hyperparameters as a result of the ensemble training procedure are presented in Table \ref{tab:BO}. In figure \ref{fig:BO} (left), we plot the exact single soliton and its PINN approximation and observe that the PINN approximate the exact solution very well. This is further verified from Table \ref{tab:BO}, where we report an error of less than $1 \%$. 
\begin{figure}[h!]
    \begin{subfigure}{.49\textwidth}
        \centering
        \includegraphics[width=1\linewidth]{{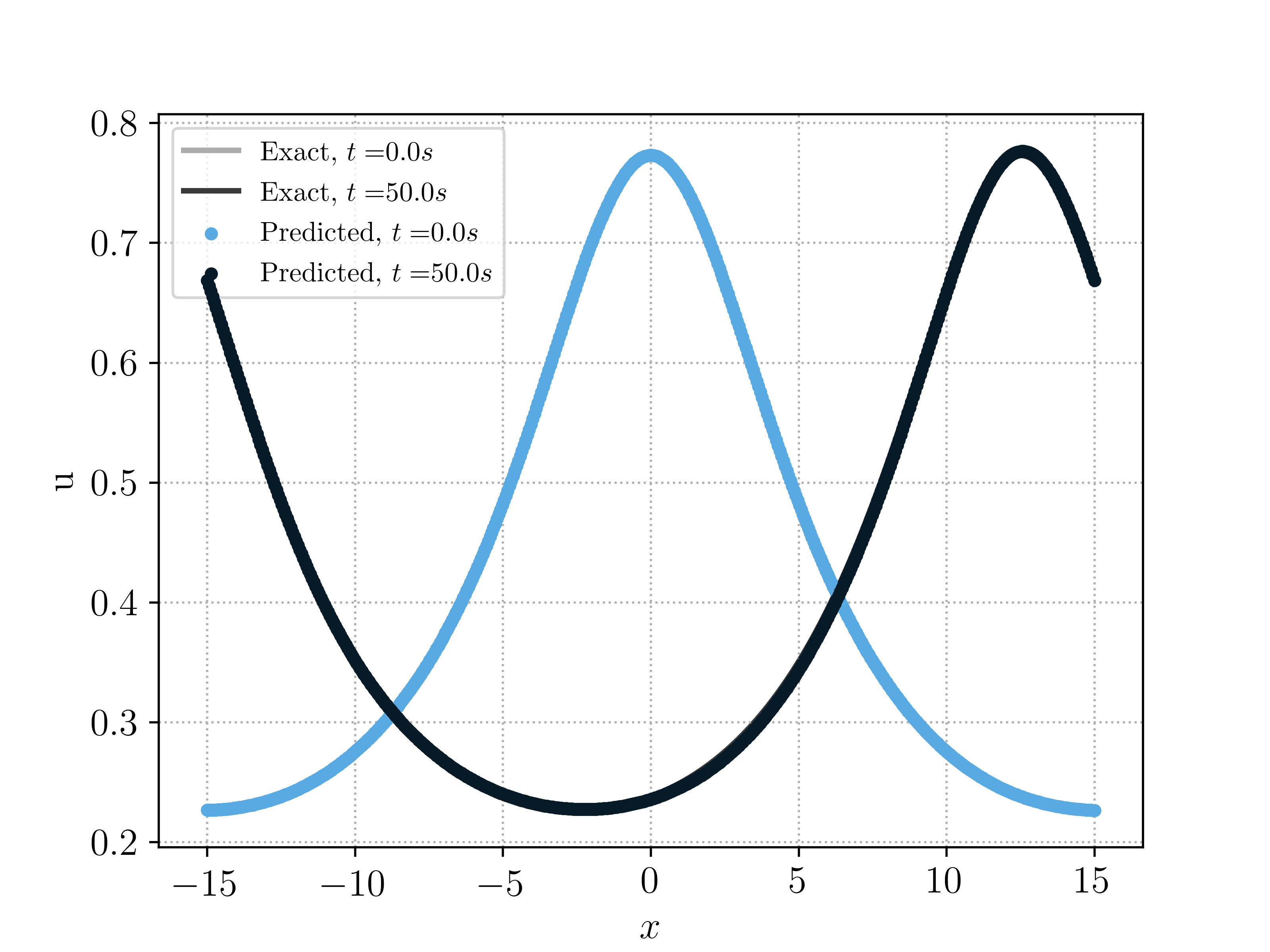}}
        \caption{Single soliton}
    \end{subfigure}
    \begin{subfigure}{.49\textwidth}
        \centering\
        \includegraphics[width=1\linewidth]{{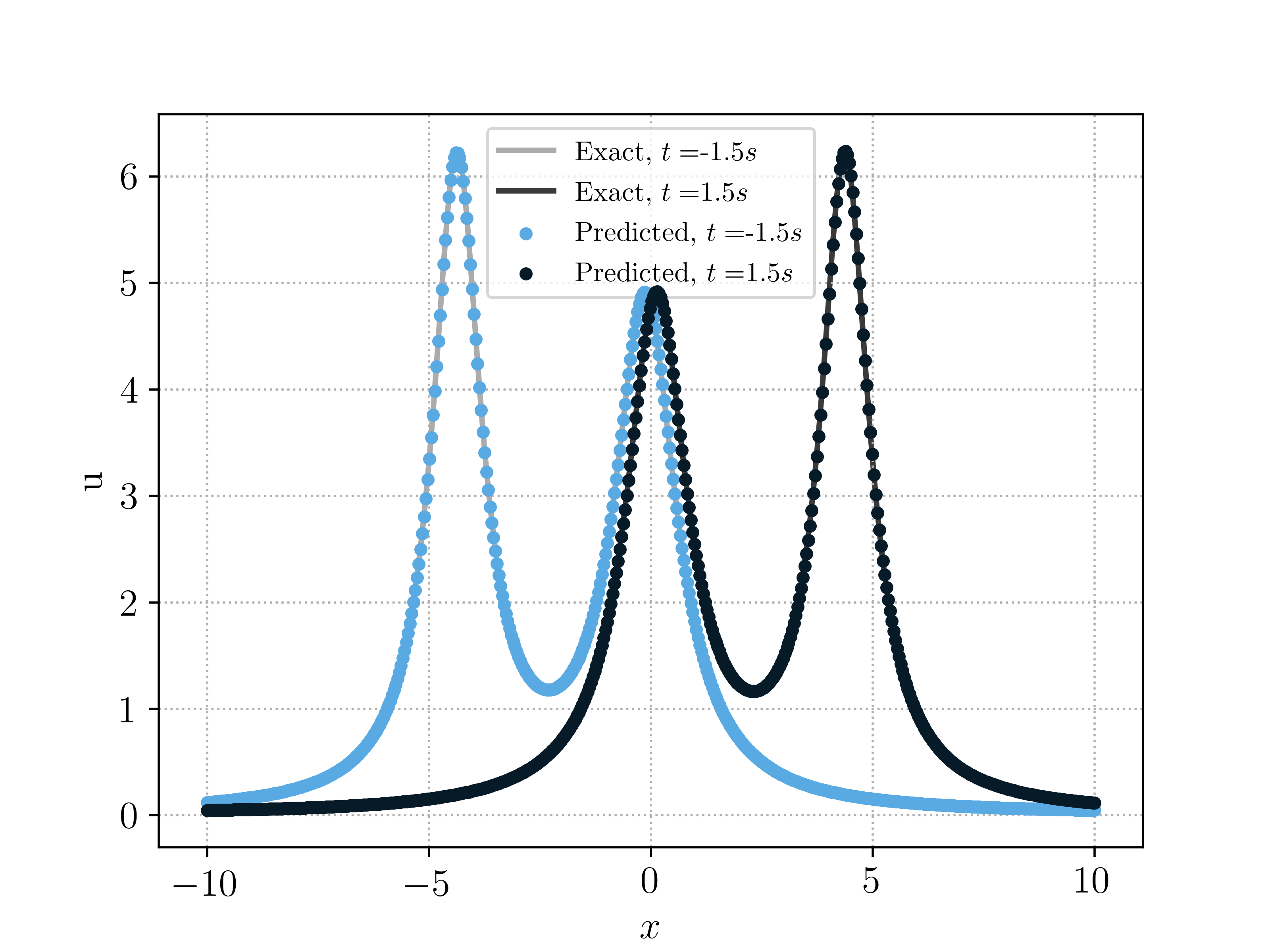}}
        \caption{Double soliton}
    \end{subfigure}
    \caption{The exact and PINN solutions of single and double soliton of BO equation.}
    \label{fig:BO}
\end{figure}
Next, we consider the double soliton case. In this case, the exact solution formula for the periodic double soliton is very complicated to evaluate. Hence, we consider the so-called \emph{long wave limit} by taking $L \rightarrow +\infty$. Hence, we consider the interacting solitons on the real line with formulas,
\begin{equation}
u(x, t) = \frac{4c_1c_2(c_1\lambda_1^2 + c_2\lambda_2^2 + (c_1+c_2)^3c_1^{-1}c_2^{-1}(c_1-c_2)^{-2})}{(c_1c_2\lambda_1\lambda_2 - (c_1+c_2)^2(c_1-c_2)^{-2})^2 + (c_1\lambda_1 + c_2\lambda_2)^2},
\end{equation}
where
\begin{equation}
\begin{aligned}
\lambda_1 = x - c_1t, \quad \lambda_2 = x - c_2t
\end{aligned}
\end{equation}
\begin{table}[htbp] 
    \centering
    \renewcommand{\arraystretch}{1.1} 
    
    \footnotesize{
        \begin{tabular}{ c c c c } 
            \toprule
            \bfseries max\_iters  &\bfseries training time$/s$ &\bfseries $\eps_T$ &\bfseries $\eps_G^r$ \\ 
            \midrule
            \midrule
         100   &    87   &  1.36e-02   &  4.11e-01 \\
         \midrule
         500   &   430   &  3.83e-03   &  2.36e-01 \\
         \midrule
        1000   &   888   &  3.30e-03   &  2.34e-01 \\
        \midrule
        2000   &  1667   &  1.61e-03  &   6.13e-02 \\
        \midrule
        5000   &  3492   &  4.56e-04  &   8.22e-03 \\
        \midrule
      10000   &  6107   &  2.96e-04   &  7.73e-03 \\
              
            \bottomrule
        \end{tabular}
    \caption{Results of different training iterations for single soliton case of BO equation.}
		\label{tab:BO_single_cp}
    }
\end{table}
This solution represents two waves that “collide” at $t = 0$ and separate for $t > 0$. For large $|t|$, $u(\cdot, t)$ is close to a sum of two single solitons at different locations. We choose $c_1 = 2 $ and $ c_2 = 1$ in our experiments. Given the impossibility of computing over the whole real line, we restrict ourselves to the computational domain $[-L, L]$. We first extend the PINN by zero to the extended computational domain $[-5L, 5L]$ and then use a similar discretization as in \eqref{eq:ehr}, to compute the discrete periodic convolution of $\frac{1}{\pi x_j}$ and $\left.u_{xx}(x)\right\vert_{x_j}$ and finally restrict the result of discrete periodic convolution onto domain $[-L, L]$.
\begin{table}[htbp] 
    \centering
    \renewcommand{\arraystretch}{1.1} 
    
    \footnotesize{
        \begin{tabular}{ c c c c } 
            \toprule
            \bfseries max\_iters  &\bfseries training time$/s$ &\bfseries $\eps_T$ &\bfseries $\eps_G^r$ \\ 
            \midrule
            \midrule
         100    &   74  &   2.98e-01  &   4.69e-01 \\
         \midrule
         500   &   325  &  3.07e-02   &  2.96e-02 \\
         \midrule
        1000   &   703   &  1.13e-02   &  3.92e-03 \\
        \midrule
        2000   &  1280  &   7.19e-03  &   6.98e-03 \\
        \midrule
        5000  &   1715   &  6.16e-03  &   6.57e-03 \\
        \midrule
      10000  &   1937   &  6.16e-03   &  6.57e-03 \\
              
            \bottomrule
        \end{tabular}
    \caption{Results of different training iterations for double soliton case of BO equation.}
		\label{tab:BO_double_cp}
    }
\end{table}

The resulting PINN approximation together with the exact double soliton is plotted in Figure \ref{fig:BO} (right). We observe a very accurate approximation of the BO double-soliton interaction by the PINN and this is also confirmed by a very low error of less than $1\%$, reported in Table \ref{tab:BO}. 

The training times for the periodic single soliton are shown in Table \ref{tab:BO_single_cp} and we see that the training is significantly slower in this case, when compared to other test cases, with a relative error of approximately $6\%$ in approximately $25$ minutes. On the other hand, the PINN approximating the real-line double soliton is significantly faster to train. From the training times reported in Table \ref{tab:BO_double_cp}, we see that an error of about $3\%$ is already achieved for a training of merely $5$ minutes. Given the non-local as well as dispersive nature of the underlying solutions, attaining such low errors in a short time is noteworthy.

\begin{figure}
	\begin{subfigure}{.49\textwidth}
		\centering
		\includegraphics[width=1\linewidth]{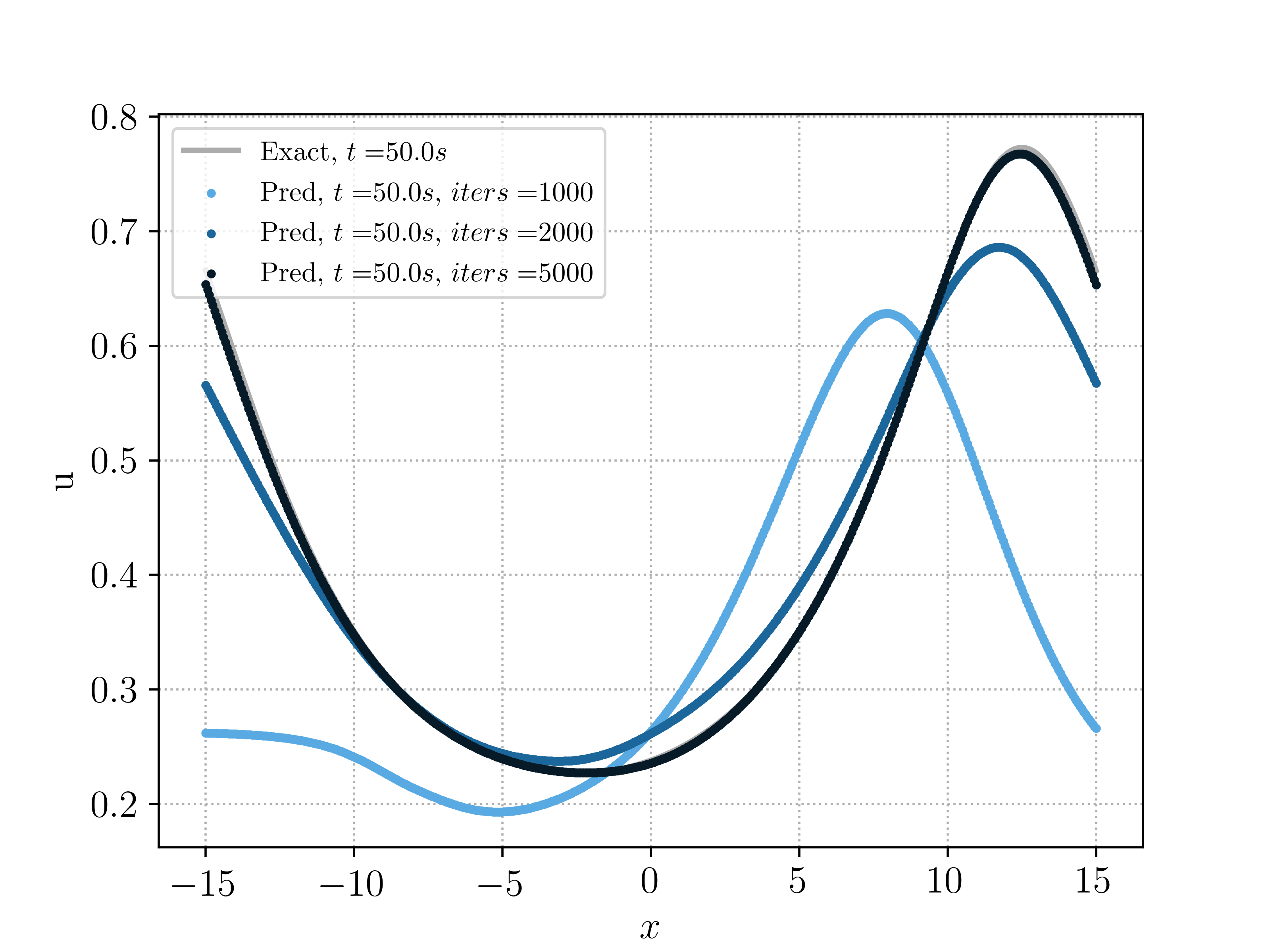}
		\caption{Periodic single soliton}
	\end{subfigure}
	\begin{subfigure}{.49\textwidth}
		\centering
		\includegraphics[width=1\linewidth]{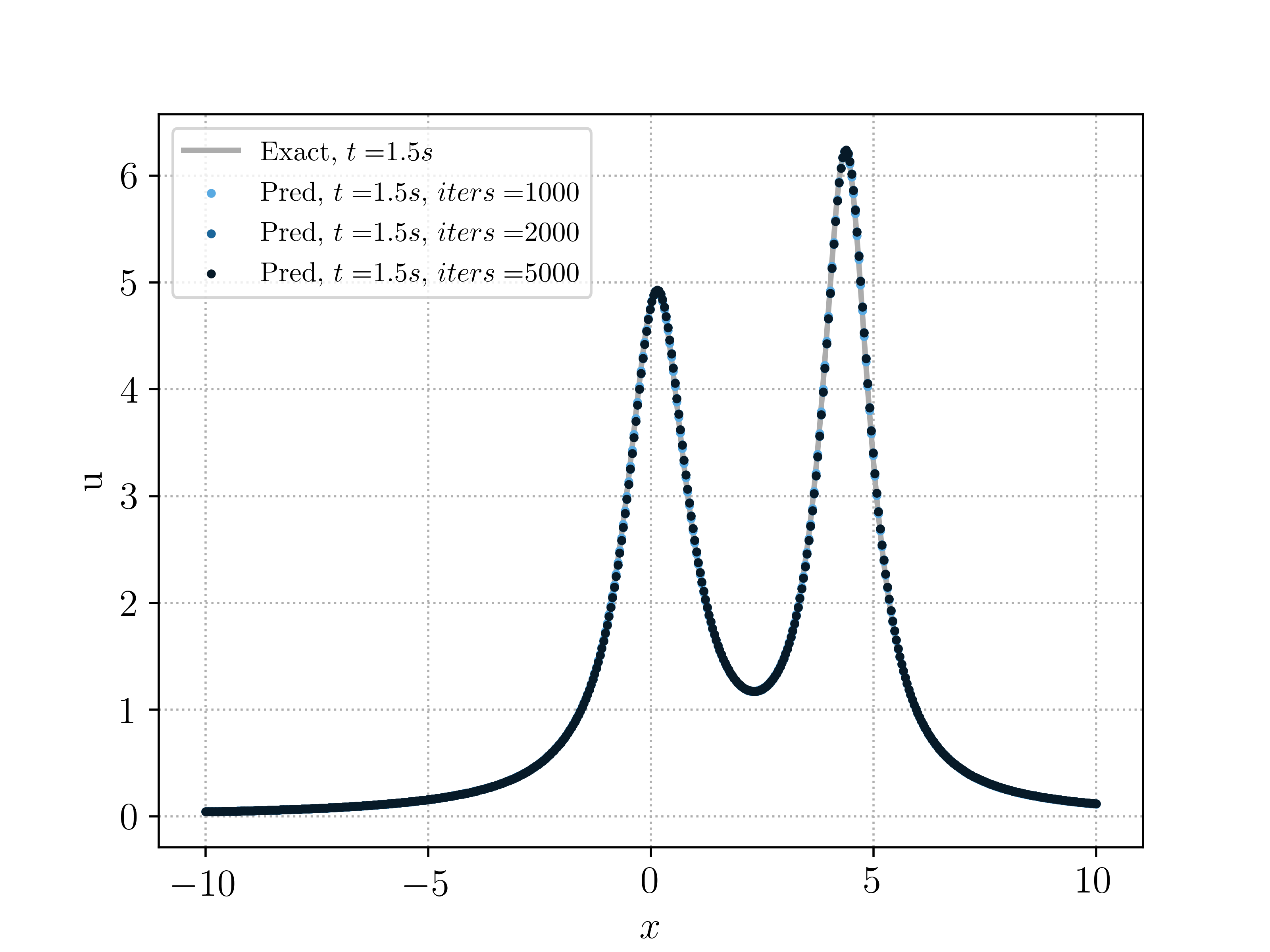}
		\caption{Real-line double soliton}
	\end{subfigure}
	\caption{Plots of different train iterations at final time.}
	\label{fig:BO_double_cp}
\end{figure}

Why is it significantly harder to train the periodic single soliton, when compared to the real-line double soliton. To investigate this question, we plot the PINN approximation to both test cases in Figure \ref{fig:BO_double_cp} for different training iterations. As observed in this figure, both the boundary values and the peak of the periodic single soliton take quite some LBFGS iterations to converge, which explains the relatively high computational cost. On the other hand, the real-line double soliton is approximated very fast as it has two sharp peaks, which are resolved with very few LBFGS iterations. 
\section{Discussion}
Nonlinear dispersive PDEs such as the KdV-Kawahara equation, the Camassa-Holm equation and the Benjamin-Ono equation arise in the modeling of shallow-water waves. In addition to being completely integrable, these PDEs contain interesting solutions such as multiple colliding solitons, which result from a subtle balance between the effects of non-linearity and dispersion. Given the fact that these PDEs are nonlinear and contain either high-order or non-local partial derivatives, standard numerical methods such as finite difference and finite element methods can be very expensive for computing accurate solutions. 

In this paper, we have proposed a novel machine learning algorithm for approximating the solutions of the afore-mentioned dispersive PDEs. Our algorithm is based on recently proposed physics informed neural networks (PINNs), in which the PDE residual, together with initial and boundary data mismatches, is minimized by a gradient descent algorithm to yield a neural network that can approximate classical solutions of the underlying PDE. We prove rigorous bounds on the error of the PINNs and present several numerical experiments to demonstrate that PINNs can efficiently approximate the solutions of non-linear dispersive equations such as KdV-Kawahara, Camassa-Holm and Benjamin-Ono. We observe from the numerical experiments that PINNs can yield very low errors with low to moderate computational cost, even for complicated problems such as multi-soliton interactions, making them significantly more efficient than traditional numerical methods for these nonlinear PDEs. Moreover, we also showed that PINNs can efficiently approximate high-dimensional parametric dispersive PDEs, which arises in the context of UQ. Finally, PINNs are very easy to code and parallelize using standard machine learning frameworks such as PyTorch and Tensorflow. 

This impressive performance of PINNs is in spite of the fact that the basis of the PINNs algorithm is an \emph{automatic differentiation by backpropagation} routine, by which one evaluates the derivatives used in computing the PDE residual. Given that one has to repeatedly use automatic differentiation for evaluating the high-order derivatives for dispersive PDEs, for instance 3rd-order derivatives for the KdV and Camassa-Holm equation and even a 5th-order derivative for the Kawahara equation, it is surprising that the automatic differentiation routine is both stable and very accurate, resulting in very low PINN errors. This paper further demonstrates the robustness of backpropagation. 

 It is clear from the error estimates that PINNs can only approximate classical solutions of dispersive PDEs efficiently. On the other hand, singular solutions such as peakons for the Camassa-Holm equation cannot be efficiently approximated by PINNs. Rather, weak formulations of PINNs will be better suited for this purpose and we plan to investigate such an extension in the future.

\bibliographystyle{abbrv}
\bibliography{Dispersive_PINN}

\end{document}